\documentclass[12pt,a4paper]{article}

\usepackage[margin=24mm]{geometry}
\usepackage{amsmath,amssymb,amsthm,comment,cite}
\usepackage{microtype}
\usepackage{xcolor}
\usepackage[hidelinks]{hyperref}

\newtheorem{theorem}{Theorem}[section]
\newtheorem{lemma}[theorem]{Lemma}
\newtheorem{proposition}[theorem]{Proposition}
\theoremstyle{remark}
\newtheorem{remark}[theorem]{Remark}
\newtheorem{example}[theorem]{Example}
\numberwithin{equation}{section}

\newcommand{\Z}{\mathbb{Z}}

\newcommand{\R}{\mathbb{R}}

\title{\textbf{Applications of the fixed point theorem for mappings contracting perimeters of triangles: ODEs, PDEs, integer and difference equations}}

\author{
    E. Petrov\thanks{Institute of Applied Mathematics and Mechanics of the NAS of Ukraine. Emails: \mbox{\url{eugeniy.petrov@gmail.com}}, \url{taranets_r@yahoo.com}}
    \and
    R. Salimov\thanks{Institute of Mathematics of the NAS of Ukraine. Email: \url{ruslan.salimov1@gmail.com}}
    \and
    R. Taranets\thanks{Yuriy Fedkovych Chernivtsi National University,
Kotsiubynskoho Str.~2,  58002 Chernivtsi, Ukraine}\ \ \footnotemark[1]
}

\date{}

\usepackage{etoolbox}
\AfterEndEnvironment{abstract}{%

\par\medskip
}

\begin{document}
\maketitle
\vspace{-2em}
\begin{abstract}
Mappings contracting perimeters of triangles are Banach contractions whenever every point of the underlying metric space is an accumulation point. Consequently, applying the perimeter-contraction theorem on a complete metric space without isolated points yields an existence result already covered by Banach's fixed point theorem. This observation applies, in particular, to operators associated with differential and integral equations acting on standard nontrivial function spaces, such as spaces of continuous functions $C[a,b]$, Lebesgue spaces $L_p$, and Sobolev spaces $W^{k,p}$, equipped with their usual norm metrics.  Thus, applications extending beyond Banach's theorem for the same operator and metric require a complete invariant domain containing some number of isolated points. Such a domain may be a closed subset of a standard function space, despite the absence of isolated points in the ambient space. The main goal of this paper is to illustrate applications of the fixed point theorem for mappings contracting perimeters of triangles in settings where the associated operators are not Banach contractions. The differential-equation examples admit interpretations as stylized models combining continuous evolution, nonlocal measurement, and quantized feedback. The integer and difference-equation examples illustrate the effects of rounding on equilibria and periodic responses.
\end{abstract}

\section{Introduction}
Let $(X,d)$ be a metric space with $|X|\geqslant 3$. In \cite{P23} it was introduced and considered a new type of mappings that contract perimeters of triangles. Recall that $T\colon X\to X$ is a \emph{mapping contracting perimeters of triangles} on $X$ if there exists $\alpha\in [0,1)$ such that the inequality
  \begin{equation}\label{e1}
   d(Tx,Ty)+d(Ty,Tz)+d(Tx,Tz) \leqslant \alpha (d(x,y)+d(y,z)+d(x,z))
  \end{equation}
holds for all three pairwise distinct points $x,y,z \in X$. Note that the requirement for $x,y,z\in X$ to be pairwise distinct is essential. One can see that otherwise this definition is equivalent to the definition of contraction mapping.

The following theorem is the main result of the work~\cite{P23}.
\begin{theorem}\label{t1}
Let $(X,d)$, $|X|\geqslant 3$, be a complete metric space and let $T\colon X\to X$ be a mapping contracting perimeters of triangles on $X$. Then $T$ has a fixed point if and only if $T$ does not possess periodic points of prime period $2$. The number of fixed points is at most two.
\end{theorem}
Although the proof of this theorem is based on the ideas of Banach's classical theorem, the essential difference
is that the definition of these mappings involves triplets of points instead of pairs. Moreover, an additional
condition was required to prevent the mapping $T$ from having periodic points of prime period $2$. Ordinary
contraction mappings form a proper subclass of these mappings, which allowed the classical Banach theorem (see \cite{Banach22}) to be obtained as a simple corollary.

The three-point formulation of perimeter contractivity has
subsequently been developed in several directions. In particular,
three-point analogues of the Kannan and Chatterjea conditions
have been investigated in~\cite{PB24Kannan,BP26Chatterjea,
PP24Chatterjea}, while enriched Kannan-type contractions in
normed spaces have been considered in~\cite{DIALGL24Enriched}.
Nonlinear extensions include $(F,G)$-contracting mappings,
an Edelstein-type fixed point theorem~\cite{BPS25FG},
and $\alpha$-$\psi$-mappings contracting triangle
perimeters~\cite{SBY25AlphaPsi}. Further generalizations
involve quadrilateral perimeter contractions~\cite{BMD25Quadrilaterals},
four-point Kannan-type mappings~\cite{BR26FourPoint},
and polygonal perimeter contractions~\cite{ZP25Polygons}.
A related approach considers contractions of the total pairwise
distances among $n$ points and investigates the existence of
periodic points~\cite{P25}. Three-point contractions involving
three metrics and their multivalued counterparts have been
studied in~\cite{JPS25Multivalued}, while connections with
$G$-metric fixed point theory have been established
in~\cite{JPaS25Gmetric}. Other extensions concern
multiplicative $\mathfrak{m}$-metric spaces~\cite{YK26Multiplicative}
and spaces equipped with multiple perturbed metrics~\cite{WDAZ26Perturbed}.
Additional developments address almost weakly Picard operators
and comparison functions~\cite{BA26Picard}, remetrization
techniques and the stability of fixed point problems~\cite{C26Remetrization},
and graph-based $p$-point orbital contractions~\cite{BBD26Graph}.
Collectively, these results demonstrate the connections between
multipoint contractivity, the multiplicity of fixed points,
periodic behavior, generalized metric structures, and
iterative approximation methods.

In~\cite{P23} it was also obtained the following important result.
\begin{proposition}\label{p1}
Let $(X,d)$, $|X|\geqslant 3$, be a metric space and let $T\colon X\to X$ be a mapping contracting perimeters of triangles. If $x$ is an accumulation point of  $X$, then inequality $d(Tx,Ty)\le\alpha d(x,y)$ holds for all points $y\in X$, where $\alpha$ is the coefficient from~(\ref{e1}). Consequently, if all points of $X$ are accumulation points, then $T$ is a standard Banach contraction.
\end{proposition}

Proposition~\ref{p1} shows that mappings contracting perimeters of triangles reduce to standard Banach contractions on spaces without isolated points, i.e., where every point is an accumulation point.
{ 
Moreover, mapping contracting perimeters of triangles that is not a Banach contraction must have at least two isolated points in its domain. This is a necessary condition. The presence of isolated points alone does not establish failure of Banach contractivity. Note also that if a fixed point of a mapping that contracts the perimeters of triangles is an accumulation point, then it is unique fixed point, see Proposition~3.7 in~\cite{P25} formulated for a more general class of mappings. In particular, this implies that if such a mapping has two distinct fixed points, they both are isolated.
}

{
\noindent\textbf{Physical background.}
Nonlocal initial and boundary conditions describe situations in which
a prescribed state is linked to measurements taken elsewhere in space
or at later times. Examples include the specification or conservation
of heat content~\cite{Cannon63,Ionkin77}, interior observations and
weighted spatial averages~\cite{BS69,Day82}, and thermostat models with
a sensor placed away from the actuator~\cite{KM04,GM21}. Nonlocal
Cauchy problems provide a related framework for initial states linked
to later observations~\cite{Byszewski91,Deng93,Pao95}. When a measurement
or actuator has finite resolution, an ideal uniform quantizer
restricts the control values to a lattice $a+h\mathbb Z$.

Quantization can alter the stability of a feedback loop. For certain
systems with fixed quantization, asymptotic stabilization may fail and
the dynamics may even be chaotic~\cite{Delchamps90}. Appropriately
adjusted quantizers can also permit global asymptotic
stabilization~\cite{BL00}. Rounding in digital recursive filters may
produce dead bands and limit cycles~\cite{OS10}, and it can change the
dynamics of spatially discretized mappings~\cite{DK93}.

The evolving state may remain continuous even when the control values
lie on a lattice. In the differential-equation examples below, each
lattice value determines a unique trajectory or elliptic solution.
The separation estimates proved below make this family a complete,
uniformly separated subset of the ambient function space. This
provides the discrete invariant domain on which we apply
Theorem~\ref{t1} and verify failure of Banach contractivity. The physical
interpretations use dimensionless variables and ideal quantizers;
they describe stationary fields or consistent initializations of
repeated operating cycles.
\par}

Let $\tau>0$ and let $f\colon\mathbb R\to\mathbb R$ be globally
Lipschitz with constant $L\ge0$. In Section~\ref{sec:general-nonlocal}, we consider a problem
\begin{equation}
\label{eq:general-problem}
u'(t)=f(u(t)),\quad 0\le t\le\tau,
\qquad u(0)=\Phi[u],
\end{equation}
where $\Phi\colon C([0,\tau],\mathbb R)\to\mathbb R$ is a functional of special type.
A solution is understood to belong to $C^1([0,\tau],\mathbb R)$.
Global Lipschitz continuity ensures, by the  {Picard--Lindel\"{o}f} theorem \cite{Picard1891, Lindelof1894} (see also~\cite[Chapter~2]{T12} for the global version), a unique solution for each
prescribed initial value, but it does not by itself ensure compatibility
with an arbitrary functional $\Phi$. For example, when $f=0$ and
$\Phi[u]=u(\tau)+1$, there is no solution; when $f=0$ and
$\Phi[u]=u(\tau)$, every constant function is a solution.
Thus the nonlocal condition needs separate hypotheses.
 {Nonlocal Cauchy problems of this kind were studied systematically by
Byszewski~\cite{Byszewski91} for abstract evolution equations. A distinctive
feature of the problems considered here is that $\Phi$ contains a
quantizer, so that the admissible initial values form a lattice.}

In Section~\ref{sec:pde-examples}, we extend this approach to
partial differential equations. We consider two nonlinear problems,
one elliptic and one parabolic, with nonlocal boundary and initial
conditions, respectively. Such conditions arise naturally in
mathematical models involving spatially distributed measurements
and feedback~\cite{Cannon63,BS69,Ionkin77,Day82,KM04,GM21}, and may provide a more realistic description of
certain physical processes than purely local conditions; nonlocal
initial conditions for parabolic equations are discussed
in~\cite{Deng93,Pao95}.
The resulting problems may be ill-posed in the sense of Hadamard,
since uniqueness of solutions is not guaranteed. In particular,
both examples admit two distinct solutions for suitable values
of their parameters. An important feature of our approach is that the
reduction to a discrete invariant set is exact:
every solution of the original nonlocal problem
belongs to the constructed family, and the solutions
of this problem coincide precisely with the fixed
points of the induced operator. Thus, no solutions
are lost in passing to the discrete setting.

Integer equations of the form $x=\lfloor f(x)\rfloor$
arise naturally in discrete dynamical systems involving
rounding and spatial discretization, where rounding may
significantly affect the existence of fixed points and
the long-term behavior of trajectories, see~\cite{DK93}.
In particular, fixed points and trajectories of the mapping
$x\mapsto\lfloor\lambda x\rfloor$ have been investigated
in~\cite{RSU16}, while periodic points and limiting behavior
of the more general affine mapping
$x\mapsto\lfloor\lambda x+\mu\rfloor$
have been studied in~\cite{H25}.
In Section~\ref{sec:integer-equations}, we apply the theory
of mappings contracting perimeters of triangles to such
integer equations. We establish local criteria for perimeter
contractivity and derive an explicit upper bound on the
number of iterations required to reach a fixed point.
Our results also apply to nonlinear integer equations.

Fixed point methods provide a useful framework for studying the
existence, uniqueness, and multiplicity of solutions of difference
equations. General background on difference equations can be found
in~\cite{DE-Elaydi05,DE-Agarwal00}.
The operator approach consists in expressing the desired solution
as a fixed point of a mapping on a suitable space of sequences.
For instance, the combination of the theory of linear difference
equations with Banach's fixed point theorem has been used
in~\cite{DE-Stevic17} to establish sufficient conditions for the
existence of a unique bounded solution of a nonhomogeneous
second-order linear recurrence. Metric fixed point theorems
have also been developed and applied to the existence of bounded
solutions of difference equations in~\cite{DE-Slyusarchuk19}.
Moreover, fixed point methods can yield multiplicity results.
In particular, discrete Green functions and a fixed point theorem
were employed in~\cite{DE-CRWZ07} to establish the existence of
multiple positive solutions of nonlinear second-order problems
with nonlocal boundary conditions.

In Section~\ref{sec5}, we apply the fixed point approach to
difference equations, considering bounded real-valued solutions
of nonautonomous scalar recurrences on the whole integer line.
We construct complete invariant subsets of $\ell_\infty(\mathbb Z)$
on which the associated operators contract perimeters of triangles.
Two examples illustrate the applicability of Theorem~\ref{t1}
in situations where Banach's contraction principle cannot
be applied.

\section{Nonlocal initial conditions for scalar differential equations}
\label{sec:general-nonlocal}
The aim of this section is to develop a fixed point approach to a class of nonlocal initial value problems for scalar differential equations. The main idea is to restrict the family of classical solutions corresponding to prescribed initial values to a suitable discrete subset of the ambient function space. In contrast to the full space $C([0,\tau],\mathbb R)$, such a subset consists of isolated points and therefore provides a setting in which mappings contracting perimeters of triangles need not reduce to Banach contractions.  {For the functionals considered below, we verify that the induced operators on these discrete sets contract perimeters of triangles and give conditions under which they fail to be Banach contractions in the supremum metric.}


Let $\tau>0$ and let $f\colon\mathbb R\to\mathbb R$ be
nonincreasing and globally Lipschitz with constant $L\ge0$.
For each $r\in\mathbb R$, the classical existence and uniqueness
theorem (see, e.\,g. \cite{T12}, Chapter~2) provides a unique solution
$Y_r\in C^1([0,\tau],\mathbb R)$ of
\begin{equation}
\label{eq:cauchy-family}
Y_r'(t)=f(Y_r(t)),\qquad Y_r(0)=r.
\end{equation}
Equip $C([0,\tau],\mathbb R)$ with the metric
\[
d(u,v)
=\max_{0\le t\le\tau}|u(t)-v(t)|,
\]
and write
\[
P_d(u,v,w)=d(u,v)+d(v,w)+d(w,u).
\]

The following elementary observation describes the metric geometry of the family of Cauchy trajectories and will be used throughout the section.

\begin{lemma}\label{lem:flow}
Let $\tau>0$ and let $f\colon\mathbb R\to\mathbb R$ be
nonincreasing and globally Lipschitz with constant $L\ge0$.
For each $r\in\mathbb R$, let
$Y_r\in C^1([0,\tau],\mathbb R)$ denote the unique solution
of \eqref{eq:cauchy-family}.
Then, for every $r<s$ and every $t\in[0,\tau]$,
\begin{equation}\label{eq:flow-bounds}
0<e^{-Lt}(s-r)\le Y_s(t)-Y_r(t)\le s-r.
\end{equation}
Consequently,
\begin{equation}\label{eq:flow-isometry}
d(Y_r,Y_s)=|r-s|
\qquad(r,s\in\mathbb R).
\end{equation}
\end{lemma}

\begin{proof}
Fix $r<s$ and put
\begin{equation}\label{eq:flow-difference}
\delta(t)=Y_s(t)-Y_r(t),\qquad 0\le t\le\tau.
\end{equation}
By \eqref{eq:cauchy-family},
\begin{equation}\label{eq:flow-derivative}
\delta(0)=s-r>0,\qquad
\delta'(t)=f(Y_s(t))-f(Y_r(t)).
\end{equation}

We first show that
\begin{equation}\label{eq:flow-positivity}
\delta(t)>0\qquad(0\le t\le\tau).
\end{equation}
Indeed, suppose that $\delta(t_0)=0$ for some
$t_0\in(0,\tau]$. Then $Y_s(t_0)=Y_r(t_0)$.
Since $f$ is globally Lipschitz, uniqueness for the
Cauchy problem holds in both time directions  {(the backward problem
$z'=-f(z)$ has the same Lipschitz constant; see~\cite[Chapter~2]{T12})}. 
Consequently, the two solutions must coincide on
$[0,t_0]$. In particular, $Y_s(0)=Y_r(0)$,
contrary to $s>r$. Thus \eqref{eq:flow-positivity}
holds.

Since $f$ is nonincreasing, \eqref{eq:flow-positivity}
implies
\[
f(Y_s(t))\le f(Y_r(t)).
\]
Moreover, by the Lipschitz condition,
\[
|f(Y_s(t))-f(Y_r(t))|
\le L|Y_s(t)-Y_r(t)|
=L\delta(t).
\]
Combining these inequalities with
\eqref{eq:flow-derivative}, we obtain
\begin{equation}\label{eq:flow-differential}
-L\delta(t)\le\delta'(t)\le0,
\qquad 0\le t\le\tau.
\end{equation}
The right-hand inequality in
\eqref{eq:flow-differential} shows that $\delta$
is nonincreasing. Hence
\begin{equation}\label{eq:flow-upper}
\delta(t)\le\delta(0)=s-r.
\end{equation}
On the other hand, the left-hand inequality in
\eqref{eq:flow-differential} gives
\[
\delta'(t)+L\delta(t)\ge0.
\]
Multiplying by $e^{Lt}$, we find
\begin{equation}\label{eq:flow-integrating-factor}
\frac{d}{dt}\bigl(e^{Lt}\delta(t)\bigr)
=e^{Lt}\bigl(\delta'(t)+L\delta(t)\bigr)
\ge0.
\end{equation}
Thus $e^{Lt}\delta(t)$ is nondecreasing, and
\begin{equation}\label{eq:flow-lower}
e^{Lt}\delta(t)\ge\delta(0)=s-r,
\qquad 0\le t\le\tau.
\end{equation}
Equivalently,
\[
\delta(t)\ge e^{-Lt}(s-r).
\]
Together with \eqref{eq:flow-upper}, this proves
\eqref{eq:flow-bounds}.

It remains to establish \eqref{eq:flow-isometry}.
By \eqref{eq:flow-bounds}, the function
$Y_s(t)-Y_r(t)$ is positive, bounded above
by $s-r$, and attains this value at $t=0$.
Hence
\[
d(Y_r,Y_s)
=\max_{0\le t\le\tau}|Y_s(t)-Y_r(t)|
=s-r.
\]
The case $r>s$ follows by symmetry, while
for $r=s$ the equality is immediate.
Thus \eqref{eq:flow-isometry} holds for all
$r,s\in\mathbb R$.
\end{proof}

\noindent\textbf{Nonlocal initial value problem.}
Let $a,b\in\mathbb R$, $h>0$, and let
$\mathcal H\colon C([0,\tau],\mathbb R)\to\mathbb R$
be a functional. Consider the nonlocal initial
value problem
\begin{equation}\label{eq:lattice-problem}
u'(t)=f(u(t)),\quad 0\le t\le\tau,
\qquad
u(0)=a+h\left\lfloor
b+\frac{\mathcal H[u]}h
\right\rfloor.
\end{equation}
Here $\lfloor x\rfloor$ denotes the greatest
integer not exceeding $x$. A solution of
\eqref{eq:lattice-problem} is a function
$u\in C^1([0,\tau],\mathbb R)$ satisfying
both relations.

For each $n\in\mathbb Z$, set
\[
u_n=Y_{a+hn},
\]
where $Y_r$ is the unique solution of
\eqref{eq:cauchy-family}, and define
\[
X_h=\{u_n:n\in\mathbb Z\}.
\]
Define the mapping $F_b\colon\mathbb Z\to\mathbb Z$
and the operator $T_b\colon X_h\to X_h$ by
\begin{equation}\label{fbn}
T_bu_n=u_{F_b(n)},  \, \text{ where } \, F_b(n)=\left\lfloor
b+\frac{\mathcal H[u_n]}h
\right\rfloor.
\end{equation}

\begin{proposition}\label{prop:fixed-points}
The set of solutions to the nonlocal initial value problem
\eqref{eq:lattice-problem} coincides exactly with the set of fixed
points of the operator \(T_b\) on the space \(X_h\).
\end{proposition}

\begin{proof}
We first describe the metric structure of $X_h$. By Lemma~\ref{lem:flow}, we have
\begin{equation}\label{eq:lattice-distance}
\begin{aligned}
d(u_m,u_n)
=d(Y_{a+hm},Y_{a+hn})
=|(a+hm)-(a+hn)|
=h|m-n|,
\qquad m,n\in\mathbb Z.
\end{aligned}
\end{equation}
Thus the correspondence
$
a+hn\longmapsto u_n
$
is an isometry from $a+h\mathbb Z$ onto $X_h$.
Since $a+h\mathbb Z$ is a closed subset of $\mathbb R$,
the metric space $(X_h,d)$ is complete.
Moreover, \eqref{eq:lattice-distance} shows that $d(u_m,u_n)\ge h>0$ whenever $m\ne n$. Hence every point of $X_h$ is isolated.

Using \eqref{eq:lattice-distance}, we obtain
\begin{equation}\label{eq:ode-fixed-distance}
\begin{aligned}
d(T_bu_n,u_n)
=d(u_{F_b(n)},u_n)
=h|F_b(n)-n|.
\end{aligned}
\end{equation}
Since $h>0$, it follows that
\begin{equation}\label{eq:ode-fixed-criterion}
T_bu_n=u_n
\quad\Longleftrightarrow\quad
F_b(n)=n.
\end{equation}
We shall use this equivalence to establish the
correspondence between solutions of
\eqref{eq:lattice-problem} and fixed points of $T_b$.

First, let $u\in C^1([0,\tau],\mathbb R)$ be a solution
of \eqref{eq:lattice-problem}.
By the nonlocal initial condition,
\begin{equation}\label{eq:ode-initial-condition}
u(0)=a+h\left\lfloor
b+\frac{\mathcal H[u]}h
\right\rfloor.
\end{equation}
Therefore,
there exists $n\in\mathbb Z$ such that
\begin{equation}\label{eq:ode-initial-index}
u(0)=a+hn.
\end{equation}
On the other hand, by the definition of $u_n$,
\begin{equation}\label{eq:ode-cauchy-trajectory}
u_n'(t)=f(u_n(t)),
\qquad
u_n(0)=a+hn.
\end{equation}
Since $u$ also satisfies the same differential
equation and, by \eqref{eq:ode-initial-index},
has the same initial value as $u_n$, uniqueness
of the Cauchy problem yields
\begin{equation}\label{eq:ode-solution-identity}
u(t)=u_n(t)
\qquad(0\le t\le\tau).
\end{equation}
In particular, $u\in X_h$.

Substituting \eqref{eq:ode-solution-identity}
into \eqref{eq:ode-initial-condition}
and using \eqref{eq:ode-initial-index}, we get
\[
a+hn
=a+h\left\lfloor
b+\frac{\mathcal H[u_n]}h
\right\rfloor.
\]
Subtracting $a$ and dividing by $h>0$, we obtain
\begin{equation}\label{eq:ode-index-equation}
n=\left\lfloor
b+\frac{\mathcal H[u_n]}h
\right\rfloor
=F_b(n).
\end{equation}
By \eqref{eq:ode-fixed-criterion}, this means that
$T_bu_n=u_n$. Together with
\eqref{eq:ode-solution-identity}, it follows that
$u$ is a fixed point of $T_b$.

Conversely, suppose that $u_n\in X_h$ is a fixed
point of $T_b$. Then
$
T_bu_n=u_n,
$
and \eqref{eq:ode-fixed-criterion} implies
\begin{equation}\label{eq:ode-fixed-index}
F_b(n)=n.
\end{equation}
We verify that $u_n$ solves
\eqref{eq:lattice-problem}.
By its definition as a Cauchy trajectory,
$u_n$ satisfies
\[
u_n'(t)=f(u_n(t)),
\qquad 0\le t\le\tau.
\]
It remains to verify the nonlocal initial condition. Using the definition of $F_b$ and
\eqref{eq:ode-fixed-index}, we have
\begin{equation}\label{eq:ode-boundary-verification}
\begin{aligned}
a+h\left\lfloor
b+\frac{\mathcal H[u_n]}h
\right\rfloor
=a+hF_b(n)
=a+hn
=u_n(0).
\end{aligned}
\end{equation}
Thus $u_n$ also satisfies the nonlocal initial
condition in \eqref{eq:lattice-problem}.
Consequently, $u_n$ is a solution of this problem.

We have proved both implications: every solution
of \eqref{eq:lattice-problem} is a fixed point of
$T_b$, and every fixed point of $T_b$ is a solution
of \eqref{eq:lattice-problem}. Therefore, the two
sets coincide.
\end{proof}

The proofs below repeatedly use several elementary
properties of the floor and ceiling functions.
For convenience, we collect these identities and
inequalities in the following proposition, which
will also be used in the subsequent applications
to integer and difference equations.

\begin{proposition}
\label{prop:rounding}
For $x,y,t\in\mathbb R$, $\eta\ge0$, any positive integer $m$,
and any integer $D\ge2$, the following relations hold:
\begin{align}
0&\le\lfloor x+\eta\rfloor-\lfloor x\rfloor
 \le\lceil\eta\rceil,\label{eq:rounding-increment}\\
\lfloor x\rfloor-\lfloor y\rfloor
 &\le\lceil x-y\rceil,\qquad
 \lfloor x\rfloor-\lfloor y\rfloor<x-y+1,
 \label{eq:rounding-difference}\\
\left\lfloor\frac{\lfloor t\rfloor}{m}\right\rfloor
 &=\left\lfloor\frac tm\right\rfloor,\qquad
 \left\lceil\frac{\lceil t\rceil}{m}\right\rceil
 =\left\lceil\frac tm\right\rceil,
 \label{eq:rounding-nested}\\
\left\lceil\frac D2\right\rceil&\le\frac{2D}{3},\qquad
\left\lceil\frac D3\right\rceil\le\frac D2.
\label{eq:rounding-bounds}
\end{align}
\end{proposition}

\begin{proof}
The first inequality follows from monotonicity and
$\lfloor x+\eta\rfloor\le\lfloor x+\lceil\eta\rceil\rfloor
=\lfloor x\rfloor+\lceil\eta\rceil$.

To prove the second line, write
$x=\lfloor x\rfloor+\alpha$ and
$y=\lfloor y\rfloor+\beta$, where
$\alpha,\beta\in[0,1)$. Since
$-1<\alpha-\beta<1$, we obtain
\[
\lfloor x\rfloor-\lfloor y\rfloor
=x-y-(\alpha-\beta)<x-y+1.
\]
As the left-hand side is an integer, it follows that
$\lfloor x\rfloor-\lfloor y\rfloor
\le\lceil x-y\rceil$.

For each integer $j$, the equivalences
$\lfloor t\rfloor\ge mj\iff t\ge mj$ and
$\lceil t\rceil\le mj\iff t\le mj$ prove the two
nested-rounding identities.

Finally, the bound for
$\lceil D/2\rceil$ follows by considering even and odd $D$.
For $D=2,3$, the bound for $\lceil D/3\rceil$ is direct;
for $D\ge4$, use $\lceil D/3\rceil\le(D+2)/3\le D/2$.
\end{proof}

\begin{lemma}
\label{lem:lattice-perimeter}
Suppose that, for some $q\in[0,1/2]$,
\begin{equation}
\label{eq:H-assumption}
0\le\mathcal H[u_m]-\mathcal H[u_n]
\le qh(m-n)
\qquad(n<m,\ n,m\in\mathbb Z).
\end{equation}
Then, for every $b\in\mathbb R$, the mapping
$T_b\colon X_h\to X_h$ has no periodic points of prime
period two and satisfies
\begin{equation}
\label{eq:lattice-perimeter}
P_d(T_bu,T_bv,T_bw)\le\frac23 P_d(u,v,w)
\end{equation}
for all pairwise distinct $u,v,w\in X_h$.
\end{lemma}

\begin{proof}
Fix $b\in\mathbb R$. We first establish an estimate
for the differences of the integer-valued function
$F_b$.
Let $n,m\in\mathbb Z$ with $n<m$. By
\eqref{eq:H-assumption},
\begin{equation}\label{eq:lp-functional}
0\le
\frac{\mathcal H[u_m]-\mathcal H[u_n]}{h}
\le q(m-n).
\end{equation}
The first inequality in \eqref{eq:lp-functional} implies
\[
\frac{\mathcal H[u_n]}{h}+b\leq \frac{\mathcal H[u_m]}{h}+b.
\]
Since the floor function is nondecreasing,~(\ref{fbn}) implies
\begin{equation}\label{eq:lp-monotonicity}
F_b(n)\le F_b(m)
\qquad(n<m).
\end{equation}
Thus $F_b$ is nondecreasing.

 {We now obtain an upper bound for $F_b(m)-F_b(n)$.
Applying \eqref{eq:rounding-increment} with}
\[
x=b+\frac{\mathcal H[u_n]}h,
\qquad
\eta=\frac{\mathcal H[u_m]-\mathcal H[u_n]}h,
\]
and using \eqref{eq:lp-functional}, we obtain
\begin{equation}\label{eq:integer-bound}
\begin{aligned}
0\le F_b(m)-F_b(n)
\le
\left\lceil
\frac{\mathcal H[u_m]-\mathcal H[u_n]}h
\right\rceil
\le\lceil q(m-n)\rceil,
\qquad n<m.
\end{aligned}
\end{equation}

We next show that $T_b$ has no periodic points
of prime period two. Suppose, to the contrary,
that there exist distinct $u_n,u_m\in X_h$ such
that
$T_bu_n=u_m$, $T_bu_m=u_n$.
By the definition of $T_b$, these equalities
are equivalent to
\begin{equation}\label{eq:lp-cycle}
F_b(n)=m,\qquad F_b(m)=n.
\end{equation}
Interchanging $n$ and $m$ if necessary, we may
assume that $n<m$. Then \eqref{eq:lp-monotonicity}
and \eqref{eq:lp-cycle} give
\[
m=F_b(n)\le F_b(m)=n,
\]
contrary to $n<m$. Therefore, $T_b$ has no
periodic points of prime period two.

It remains to establish the perimeter
contraction estimate. Let $u_i,u_j,u_k \in X_h$ be pairwise
distinct. Since the perimeter is symmetric in its arguments, we may assume
that
$
i<j<k$.
Put
\begin{equation}\label{eq:lp-D}
D=k-i.
\end{equation}
Because $i,j,k$ are distinct integers, we have $D\ge2$. By \eqref{eq:lattice-distance}, we obtain
\begin{equation}\label{eq:lp-original-perimeter}
\begin{aligned}
P_d(u_i,u_j,u_k)
&=d(u_i,u_j)+d(u_j,u_k)+d(u_k,u_i)
=h(j-i)+h(k-j)+h(k-i)\\
&=2h(k-i)
=2hD.
\end{aligned}
\end{equation}

We now calculate the perimeter of the image
triangle. By \eqref{eq:lp-monotonicity},
\begin{equation}\label{eq:lp-ordered-images}
F_b(i)\le F_b(j)\le F_b(k).
\end{equation}
Since $T_bu_n=u_{F_b(n)}$, the distance formula
\eqref{eq:lattice-distance} and
\eqref{eq:lp-ordered-images} yield
\begin{equation}\label{eq:lp-image-perimeter}
\begin{aligned}
P_d(T_bu_i,T_bu_j,T_bu_k)
&=d(u_{F_b(i)},u_{F_b(j)})
 +d(u_{F_b(j)},u_{F_b(k)})
 +d(u_{F_b(k)},u_{F_b(i)})\\
&=h\bigl(F_b(j)-F_b(i)\bigr)
+h\bigl(F_b(k)-F_b(j)\bigr)
+h\bigl(F_b(k)-F_b(i)\bigr)\\
&=2h\bigl(F_b(k)-F_b(i)\bigr).
\end{aligned}
\end{equation}
Notice that this equality remains valid even
if two or all three image points coincide.

Dividing \eqref{eq:lp-image-perimeter} by
\eqref{eq:lp-original-perimeter}, and then
applying \eqref{eq:integer-bound} with
$n=i$ and $m=k$, we obtain
\begin{equation}\label{eq:lp-ratio}
\begin{aligned}
\frac{P_d(T_bu_i,T_bu_j,T_bu_k)}
       {P_d(u_i,u_j,u_k)}
=\frac{F_b(k)-F_b(i)}{D}
\le\frac{\lceil qD\rceil}{D}
\le\frac{\lceil D/2\rceil}{D},
\end{aligned}
\end{equation}
where the last inequality follows from
$q\le1/2$.

 {Since $D\ge2$, inequality \eqref{eq:rounding-bounds} yields
$\lceil D/2\rceil/D\le2/3$. Combining this estimate
with \eqref{eq:lp-ratio}, we conclude that}
\[
P_d(T_bu_i,T_bu_j,T_bu_k)
\le\frac23 P_d(u_i,u_j,u_k),
\]
which proves \eqref{eq:lattice-perimeter}.

\end{proof}

\begin{theorem}
\label{thm:lattice}
Suppose that \eqref{eq:H-assumption} holds for some
$q\in[0,1/2]$.
Then, for every $b\in\mathbb R$, problem
\eqref{eq:lattice-problem} has at least one and at most two
solutions. If there are two solutions, their initial values
are $a+hn$ and $a+h(n+1)$ for some $n\in\mathbb Z$.
\end{theorem}

\begin{proof}
Fix $b\in\mathbb R$.
By \eqref{eq:lattice-distance}, the metric space $(X_h,d)$
is isometric to $a+h\mathbb Z$ and hence is complete.
Lemma~\ref{lem:lattice-perimeter} shows that $T_b$ contracts
perimeters of triangles on $X_h$ and has no periodic points
of prime period two. Therefore, Theorem~\ref{t1} implies
that $T_b$ has at least one and at most two fixed points.
By Proposition~\ref{prop:fixed-points}, these fixed points
are precisely the solutions of \eqref{eq:lattice-problem}.

Suppose that $u_n$ and $u_m$, where $n<m$, are two distinct
solutions. Then $F_b(n)=n$ and $F_b(m)=m$.
 {Using \eqref{eq:rounding-difference} and
\eqref{eq:H-assumption}, we obtain}
\[
\begin{aligned}
m-n
&=\left\lfloor b+\frac{\mathcal H[u_m]}h\right\rfloor
 -\left\lfloor b+\frac{\mathcal H[u_n]}h\right\rfloor\\
&<\frac{\mathcal H[u_m]-\mathcal H[u_n]}h+1
\le q(m-n)+1.
\end{aligned}
\]
Consequently,
\[
0<m-n<\frac{1}{1-q}\le2.
\]
Since $m-n$ is an integer, $m=n+1$.
Thus the initial values of the two solutions are
$a+hn$ and $a+h(n+1)$.
\end{proof}

The following observation provides a criterion for distinguishing
these applications from those of the Banach contraction principle.

\begin{remark}
\label{lem:nonconstant-shooting}
Let us show that under the hypotheses of Theorem~\ref{thm:lattice}, for each
$b\in\mathbb R$, the mapping $T_b$ is a Banach contraction in
the supremum metric if and only if $F_b$ is constant.
In particular, if
\begin{equation}
\label{eq:further-unbounded}
\mathcal H[u_n]\longrightarrow+\infty
\qquad(n\longrightarrow+\infty),
\end{equation}
then $T_b$ is not a Banach contraction in this metric for any
$b\in\mathbb R$.

Indeed, by \eqref{eq:integer-bound},
\[
F_b(n+1)-F_b(n)\in\{0,1\}\qquad(n\in\mathbb Z).
\]
If $F_b$ is nonconstant, this difference equals $1$ for some
$n$. Using \eqref{eq:lattice-distance}, we obtain
\[
d(T_bu_{n+1},T_bu_n)
=d\bigl(u_{F_b(n+1)},u_{F_b(n)}\bigr)
=h\,|F_b(n+1)-F_b(n)|=h=d(u_{n+1},u_n).
\]
Thus no Banach contraction coefficient less than one is
possible. If $F_b$ is constant, then $T_b$ is constant and is
a Banach contraction with coefficient zero. Finally,
\eqref{eq:further-unbounded} implies $F_b(n)\to+\infty$,
so $F_b$ is nonconstant for every $b$.
\end{remark}

\noindent\textbf{Sufficient conditions on the nonlocal functional.}
Condition~\eqref{eq:H-assumption} concerns only the Cauchy
trajectories belonging to $X_h$. We now give simple sufficient
conditions on $\mathcal H$ that guarantee this inequality. Suppose that $\mathcal H$ is order-preserving with respect to
the pointwise order, that is,
\[
u(t)\le v(t)\quad (0\le t\le\tau)
\quad\Longrightarrow\quad
\mathcal H[u]\le\mathcal H[v],
\]
and satisfies the Lipschitz condition
\begin{equation}
\label{eq:functional-lipschitz}
|\mathcal H[u]-\mathcal H[v]|
\le q\|u-v\|_\infty,
\qquad u,v\in C([0,\tau],\mathbb R),
\end{equation}
for some $q\in[0,1/2]$.

To verify \eqref{eq:H-assumption}, fix two integers $n<m$.
Recall that $u_n=Y_{a+hn}$ and $u_m=Y_{a+hm}$.
Applying \eqref{eq:flow-bounds} with
$r=a+hn$ and $s=a+hm$, we obtain
\[
0<e^{-Lt}h(m-n)
\le u_m(t)-u_n(t)
\le h(m-n),
\qquad 0\le t\le\tau.
\]
In particular, $u_n(t)\le u_m(t)$ for every
$t\in[0,\tau]$. Since $\mathcal H$ is
order-preserving, it follows that
\begin{equation}
\mathcal H[u_n]\le\mathcal H[u_m].
\end{equation}

On the other hand, by \eqref{eq:flow-isometry},
\[
\|u_m-u_n\|_\infty
=|(a+hm)-(a+hn)|
=h(m-n).
\]
Therefore, the Lipschitz condition
\eqref{eq:functional-lipschitz} yields
\[
\begin{aligned}
0
\le\mathcal H[u_m]-\mathcal H[u_n]
=|\mathcal H[u_m]-\mathcal H[u_n]|
\le q\|u_m-u_n\|_\infty
=qh(m-n).
\end{aligned}
\]
This is precisely \eqref{eq:H-assumption}.

Thus, order preservation and the Lipschitz condition
\eqref{eq:functional-lipschitz} on the whole space
$C([0,\tau],\mathbb R)$ are sufficient for applying
Theorem~\ref{thm:lattice}. They are not necessary, however:
the theorem only requires the corresponding inequalities
on the discrete family $X_h$. Example~\ref{ex:signed-multipoint}
illustrates this distinction.

\begin{proposition}
\label{ex:further-nonlocal-data}
Let $\tau>0$, $a\in\mathbb R$, $h>0$, and let
$f\colon\mathbb R\to\mathbb R$ be nonincreasing
and globally Lipschitz. For each $n\in\mathbb Z$,
let $u_n=Y_{a+hn}$ be the corresponding Cauchy
trajectory. The following functionals satisfy
\eqref{eq:H-assumption} under the stated bounds.

\begin{enumerate}
\item[(i)] Let
\[
\mathcal H[u]=\sum_{j=1}^{N}c_j\psi_j(u(t_j)),
\]
where $N\ge1$, $c_j\ge0$, $t_j\in[0,\tau]$, and each
$\psi_j\colon\mathbb R\to\mathbb R$ is nondecreasing and
globally $\gamma_j$-Lipschitz. A sufficient bound is
\[
q_{\mathrm{mp}}=\sum_{j=1}^{N}c_j\gamma_j\le\frac12.
\]

\item[(ii)] Let
\begin{equation}
\label{eq:integral-H}
\mathcal H[u]=\lambda\int_0^\tau K(s)\psi(u(s))\,ds,
\end{equation}
where $\lambda\ge0$, $K\colon[0,\tau]\to[0,\infty)$ is
continuous, and $\psi$ is nondecreasing and globally
$\gamma$-Lipschitz. A sufficient bound is
\[
q_{\mathrm{int}}=\lambda\gamma\int_0^\tau K(s)\,ds
\le\frac12.
\]
\item[(iii)] Let
\[
\mathcal H[u]=\lambda\max_{0\le t\le\tau}
\{\beta(t)+\psi(u(t))\},
\]
where $\lambda\ge0$, $\beta$ is continuous, and $\psi$ is
nondecreasing and globally $\gamma$-Lipschitz. A sufficient
bound is $q_{\mathrm{max}}=\lambda\gamma\le1/2$.
\end{enumerate}
\end{proposition}

\begin{proof}
We prove that each of the three functionals satisfies
\eqref{eq:H-assumption} with the corresponding constant
$q$. Let $u,v\in C([0,\tau],\mathbb R)$ satisfy
$
u(t)\le v(t)$, $0\le t\le\tau$,
and put
$
\delta=\|v-u\|_\infty
=\max_{0\le t\le\tau}|v(t)-u(t)|$.
We shall show that, in each case,
\begin{equation}\label{eq:fnd-general}
0\le \mathcal H[v]-\mathcal H[u]
\le q\delta,
\end{equation}
where $q$ is the constant specified in the
corresponding part of the proposition.

\smallskip
\noindent\emph{(i)} Since each $\psi_j$ is nondecreasing and
$\gamma_j$-Lipschitz, we have
\[
0\le
\psi_j(v(t_j))-\psi_j(u(t_j))
\le \gamma_j|v(t_j)-u(t_j)|
\le\gamma_j\delta.
\]
Multiplying by $c_j\ge0$ and summing over
$j=1,\ldots,N$, we obtain
\begin{align}
0
\le \mathcal H[v]-\mathcal H[u]
=\sum_{j=1}^N c_j
\bigl(\psi_j(v(t_j))-\psi_j(u(t_j))\bigr)
\le\sum_{j=1}^N c_j\gamma_j\delta
=q_{\mathrm{mp}}\delta.
\label{eq:fnd-multipoint}
\end{align}
Thus \eqref{eq:fnd-general} holds with
$q=q_{\mathrm{mp}}\le1/2$.

\smallskip
\noindent\emph{(ii)} The monotonicity and Lipschitz continuity of
$\psi$ imply that, for every $s\in[0,\tau]$,
\begin{equation}\label{eq:fnd-psi}
0\le\psi(v(s))-\psi(u(s))
\le\gamma|v(s)-u(s)|
\le\gamma\delta.
\end{equation}
Since $\lambda\ge0$ and $K(s)\ge0$, multiplying
\eqref{eq:fnd-psi} by $\lambda K(s)$ and
integrating over $[0,\tau]$ gives
\begin{align}
0
\le\mathcal H[v]-\mathcal H[u]
=\lambda\int_0^\tau K(s)
\bigl(\psi(v(s))-\psi(u(s))\bigr)\,ds
\le\lambda\gamma\delta
\int_0^\tau K(s)\,ds
=q_{\mathrm{int}}\delta.
\label{eq:fnd-integral}
\end{align}
Hence \eqref{eq:fnd-general} holds with
$q=q_{\mathrm{int}}\le1/2$.

\smallskip
\noindent\emph{(iii)} Define
\[
A(t)=\beta(t)+\psi(u(t)),
\qquad
B(t)=\beta(t)+\psi(v(t)).
\]
Since $\psi$ is nondecreasing and
$\gamma$-Lipschitz, we have
\begin{equation}\label{eq:fnd-max-pointwise}
0\le B(t)-A(t)\le\gamma\delta
\qquad(0\le t\le\tau).
\end{equation}
In particular,
\[
A(t)\le B(t)\le A(t)+\gamma\delta
\qquad(0\le t\le\tau).
\]

Taking the maximum over $[0,\tau]$ in
these inequalities, we obtain
\begin{equation}\label{eq:fnd-max-inequality}
\max_{0\le t\le\tau}A(t)
\le\max_{0\le t\le\tau}B(t)
\le\max_{0\le t\le\tau}A(t)+\gamma\delta.
\end{equation}
Here the second inequality follows from the
fact that adding a constant to a function
increases its maximum by the same constant.

Multiplying \eqref{eq:fnd-max-inequality}
by $\lambda\ge0$ and using the definition
of $\mathcal H$, we conclude that
\begin{equation}\label{eq:fnd-maximum}
0\le\mathcal H[v]-\mathcal H[u]
\le\lambda\gamma\delta
=q_{\mathrm{max}}\delta.
\end{equation}
Thus \eqref{eq:fnd-general} holds with
$q=q_{\mathrm{max}}\le1/2$.

We have established \eqref{eq:fnd-general}
for all three functionals, with the corresponding
constants
$q_{\mathrm{mp}}$, $q_{\mathrm{int}}$,
and $q_{\mathrm{max}}$.
It remains to verify that this estimate
implies \eqref{eq:H-assumption} for the
Cauchy trajectories $u_n$.
Fix two integers $n<m$. Recall that
$
u_n=Y_{a+hn}$ and
$u_m=Y_{a+hm}$.
By Lemma~\ref{lem:flow}, in particular
\eqref{eq:flow-bounds}, we have
$
u_n(t)<u_m(t)$, 
$0\le t\le\tau$.
Moreover, by \eqref{eq:flow-isometry},
\begin{equation}\label{eq:fnd-flow-distance}
\|u_m-u_n\|_\infty
=d(u_m,u_n)
=|(a+hm)-(a+hn)|
=h(m-n).
\end{equation}

We may therefore apply
\eqref{eq:fnd-general} with $u=u_n$
and $v=u_m$. Using
\eqref{eq:fnd-flow-distance}, we obtain
\[
0\le\mathcal H[u_m]-\mathcal H[u_n]
\le q\|u_m-u_n\|_\infty
=qh(m-n).
\]
This is precisely condition
\eqref{eq:H-assumption}.
\end{proof}

\begin{example}[A multipoint condition with coefficients of opposite signs]
\label{ex:signed-multipoint}
Consider
\begin{equation}
\label{eq:signed-multipoint-problem}
u'(t)=-u(t),\quad 0\le t\le1,
\qquad
u(0)=\left\lfloor b+u(1/2)-u(1)\right\rfloor,
\end{equation}
where $b\in\mathbb R$. Here $a=0$, $h=1$, and
$\mathcal H[u]=u(1/2)-u(1)$. The trajectories with integer
initial values are
\[
u_n(t)=ne^{-t},\qquad n\in\mathbb Z.
\]
Consequently,
\[
\mathcal H[u_m]-\mathcal H[u_n]=c(m-n),
\qquad
c=e^{-1/2}-e^{-1}.
\]
Writing $x=e^{-1/2}\in(0,1)$, we have $0<c=x(1-x)\le1/4$.
Thus \eqref{eq:H-assumption} holds with $q=c$, and
Theorem~\ref{thm:lattice} yields at least one and at most two
solutions of \eqref{eq:signed-multipoint-problem} for every
$b$. 

For
$
b_*=1-\frac c2,
$
there are exactly two solutions,
$
u_0(t)=0$, and  $u_1(t)=e^{-t}$.
Indeed,
\[
F_{b_*}(0)=\left\lfloor1-\frac c2\right\rfloor=0,
\qquad
F_{b_*}(1)=\left\lfloor1+\frac c2\right\rfloor=1.
\]
Theorem~\ref{thm:lattice} excludes further solutions.
Moreover, $\mathcal H[u_n]=cn\to+\infty$, so
Remark~\ref{lem:nonconstant-shooting} rules out Banach
contractivity in the supremum metric for every $b$.
At $b=b_*$, the two distinct fixed points rule out Banach
contractivity with respect to any metric on $X_1$.

This example also shows that order preservation of
$\mathcal H$ on the ambient function space is not required
by Theorem~\ref{thm:lattice}. For instance,
$v(t)=0\le w(t)=(2t-1)^2$, but
$\mathcal H[v]=0>\mathcal H[w]=-1$. The required monotonicity
and estimate hold on the Cauchy trajectories.

{ 
\noindent\emph{Physical interpretation.}
The equation $u'=-u$ describes first-order decay in normalized units,
as for a radioactive source or the temperature excess of a body under
Newton's law of cooling. Along each trajectory,
\[
\mathcal H[u]=u(1/2)-u(1)=\int_{1/2}^{1}u(s)\,ds.
\]
Thus $\mathcal H[u]$ represents a normalized integrated decay signal;
in the cooling interpretation, the heat released is proportional to
this quantity, with heat capacity supplying the proportionality
factor. The floor rule prescribes the integer loading of the next
independently initialized cycle from the completed cycle's signal.
At $b=b_*$, the loadings $0$ and $1$ both reproduce themselves under
this rule. The radioactive-source interpretation applies to
nonnegative loadings; the mathematical result concerns the whole
lattice.
\par}
\end{example}

\begin{example}[A nonlinear equation with an integral nonlocal condition]
\label{ex:nonlinear-nonlocal}
Let $b\in\mathbb R$. Consider
\begin{equation}
\label{eq:nonlinear-problem}
u'(t)=\frac14(1-\tanh u(t)),\quad 0\le t\le1,
\qquad
u(0)=\left\lfloor b+\frac14
\int_0^1\ln(1+e^{u(s)})\,ds\right\rfloor.
\end{equation}
Set
\[
f(r)=\frac14(1-\tanh r),\qquad
\psi(r)=\ln(1+e^r).
\]
Since
\[
f'(r)=-\frac{1}{4\cosh^2 r},\qquad
\psi'(r)=\frac{e^r}{1+e^r},
\]
$f$ is nonincreasing and globally $1/4$-Lipschitz, while
$\psi$ is increasing and globally $1$-Lipschitz.
The functional
\[
\mathcal H[u]=\frac14\int_0^1\psi(u(s))\,ds
\]
is a special case of \eqref{eq:integral-H} with $\tau=1$, $K=1$,
and $\lambda=1/4$. By Proposition~\ref{ex:further-nonlocal-data},
it satisfies \eqref{eq:H-assumption} with $q=1/4$.
Taking $a=0$ and $h=1$, Theorem~\ref{thm:lattice} shows that
\eqref{eq:nonlinear-problem} has one or two solutions for every
$b$. The corresponding mapping $T_b$ contracts perimeters of
triangles on $X_1=\{u_n=Y_n:n\in\mathbb Z\}$.

We now show that for $b=3/4$, problem
\eqref{eq:nonlinear-problem} has exactly two solutions. In view of Theorem~\ref{thm:lattice}, it suffices to verify that $u_0$ and $u_1$ satisfy the nonlocal initial condition, or, equivalently, that $F_{3/4}(0)=0$, $F_{3/4}(1)=1$. To this end, we first derive bounds on these trajectories and then estimate the corresponding integral terms. Note that
$f(r)>0$ for all $r\in\mathbb R$.
Hence each trajectory $u_n$, $n=0,1$, is increasing,
and therefore $u_n(t)\ge u_n(0)=n\ge0$.
Since $0<f(r)\le1/4$ for $r\ge0$, we obtain
\[
0\le u_n(t)-n
=\int_0^t f(u_n(s))\,ds
\le\frac{t}{4}\le\frac14.
\]
Consequently,
\begin{equation}\label{bounds}
0\le u_0(t)\le\frac14,\qquad
1\le u_1(t)\le\frac54
\qquad(0\le t\le1).
\end{equation}
For every $r\ge0$, the identity
$
\psi(r)=\ln(1+e^r)=r+\ln(1+e^{-r})
$
implies
$
r<\psi(r)\le r+\ln 2$.
Combining these inequalities with bounds~(\ref{bounds})
we obtain
\[
0<\psi(u_0(t))\le\frac14+\ln 2<1,\quad
1<\psi(u_1(t))\le\frac54+\ln 2<2,
\]
where we have used $\ln 2<3/4$.
Integrating over $[0,1]$ and writing
\[
M_n=\int_0^1\psi(u_n(s))\,ds,
\]
we conclude that
$0<M_0<1$, $1<M_1<2$.
Consequently,
\[
\frac34<
\frac34+\frac{M_0}{4}<1,
\qquad
1<
\frac34+\frac{M_1}{4}<\frac54.
\]
Taking integer parts gives
\[
F_{3/4}(0)
=\left\lfloor\frac34+\frac{M_0}{4}\right\rfloor=0,
\qquad
F_{3/4}(1)
=\left\lfloor\frac34+\frac{M_1}{4}\right\rfloor=1.
\]
Thus both trajectories satisfy the nonlocal initial
condition and hence solve~\eqref{eq:nonlinear-problem}.
Theorem~\ref{thm:lattice} excludes any further solutions.  Moreover, since $T_{3/4}$
has two distinct fixed points, it cannot
be a Banach contraction on $X_1$ with
respect to any metric.

 {Finally, $T_b$ is not a Banach contraction in the supremum metric
for any $b$. Indeed, let $n\ge1$. Since $f>0$, the trajectory $u_0$ is
increasing, so $u_0(t)\ge0$. Lemma~\ref{lem:flow} with $L=1/4$ gives
\[
u_n(t)\ge u_0(t)+e^{-t/4}n\ge e^{-1/4}n,\qquad 0\le t\le1.
\]
Since $\psi(r)>r$ for all $r\in\mathbb R$, it follows that
\[
\mathcal H[u_n]\ge\frac{e^{-1/4}}{4}\,n\longrightarrow+\infty
\qquad(n\to+\infty),
\]
and Remark~\ref{lem:nonconstant-shooting} applies.}

{ 
\noindent\emph{Physical interpretation.}
Problem~\eqref{eq:nonlinear-problem} is a stylized thermal process with
separately initialized operating cycles. The variable $u$ is a
dimensionless temperature relative to a reference value, and the
bounded, decreasing rate $f(u)=\frac14(1-\tanh u)$ models a
self-regulating heater. The sensor response
$\psi(u)=\ln(1+e^u)$ is weakly sensitive at low temperatures and
asymptotically linear at high temperatures. A digital controller uses
the time-averaged sensor signal from a completed cycle to prescribe
the next cycle's starting temperature in integer steps. Writing $n_k$
for the integer initial value in cycle $k$ and $u_{n_k}$ for its
trajectory gives
\[
n_{k+1}=\left\lfloor b+\frac14
\int_0^1\psi(u_{n_k}(s))\,ds\right\rfloor=F_b(n_k).
\]
The preparation of this initial temperature is an external operation
between cycles. A solution of the nonlocal problem corresponds to
$n_{k+1}=n_k=n_*$: the prescribed starting temperature reproduces
itself from one cycle to the next.
\par}
\end{example}

\section{Non-trivial examples for nonlinear PDEs}
\label{sec:pde-examples}

The preceding examples concern ordinary differential equations. We now give two PDE examples. As in Section~\ref{sec:general-nonlocal}, the ambient function space has no isolated points, but the nonlocal condition selects a discrete invariant solution set. This is precisely the situation in which Theorem~\ref{t1} is not reduced to Banach's theorem by Proposition~\ref{p1}.

\begin{example}[A nonlinear elliptic equation with a nonlocal boundary condition]
\label{ex:elliptic-pde}
Let $\Omega\subset\mathbb R^d$ be a bounded $C^2$ domain  {(open and connected)} and let $|\Omega|$ denote its Lebesgue measure. Set
\[
f(r)=\frac{1}{4}(1-\tanh r),\qquad
\psi(r)=\ln(1+e^r),
\]
and define
\[
\mathcal H[u]=\frac{1}{4|\Omega|}\int_\Omega \psi(u(x))\,dx.
\]
Consider the semilinear elliptic problem
\begin{equation}\label{eq:elliptic-nonlocal-problem}
\begin{cases}
-\Delta u=f(u), & x\in\Omega,\\[2mm]
u=\left\lfloor b+\mathcal H[u] \right\rfloor, & x\in\partial\Omega,
\end{cases}
\end{equation}
where $b\in\mathbb R$ and the right-hand side of the boundary condition is understood as a constant function on $\partial\Omega$. A solution is understood as a function $u\in C^2(\Omega)\cap C(\overline\Omega)$ satisfying \eqref{eq:elliptic-nonlocal-problem} pointwise.

{ 
\noindent\emph{Physical interpretation.}
Problem~\eqref{eq:elliptic-nonlocal-problem} is a stylized model of the
stationary temperature field in a body $\Omega$ with a distributed,
self-regulating heat source of density $f(u)$. A digital thermostat
sets the surface temperature in integer steps using the volume
average of the nonlinear sensor response $\psi(u)$. The boundary condition expresses
compatibility between the stationary field and the thermostat rule.
As proved below, uniqueness for each integer boundary value and
separation in the supremum metric yield a complete discrete family
of stationary fields. For $b=b_*$, two such fields satisfy the same
feedback rule.
\par}

For each $n\in\mathbb Z$, let $U_n$ be the unique solution of
\begin{equation}\label{eq:elliptic-family}
\begin{cases}
-\Delta U_n=f(U_n), & x\in\Omega,\\[2mm]
U_n=n, & x\in\partial\Omega.
\end{cases}
\end{equation}
{ 
To justify existence, let
$\varphi\in C^2(\Omega)\cap C(\overline\Omega)$ solve
\[
\begin{cases}
-\Delta\varphi=1, & x\in\Omega,\\[2mm]
\varphi=0, & x\in\partial\Omega.
\end{cases}
\]
 {Such $\varphi$ exists and is unique, since a $C^2$ domain satisfies the
exterior sphere condition; see~\cite[Theorem~6.13]{GT}.}
The maximum principle~\cite[Theorem~3.1]{GT} gives $\varphi\ge0$. Since
$0<f(r)<1/2$ for all $r\in\mathbb R$, the functions
$\underline U_n=n$ and $\overline U_n=n+\varphi/2$ satisfy
\[
-\Delta\underline U_n=0\le f(\underline U_n),\qquad
-\Delta\overline U_n=\frac12\ge f(\overline U_n),
\]
and have the same boundary value $n$. Thus they are an ordered
subsolution and supersolution.
 {For monotone iteration, use the shifted operator
$-\Delta+\frac14$. Since $f'(r)\ge-\frac14$, the function
$r\mapsto f(r)+\frac14r$ is nondecreasing. One may therefore iterate
\[
\begin{cases}
(-\Delta+\frac14)V_{j+1}=f(V_j)+\frac14V_j, & x\in\Omega,\\[1mm]
V_{j+1}=n, & x\in\partial\Omega,
\end{cases}
\]
starting from the subsolution or the supersolution.}
The subsolution--supersolution method
and elliptic regularity yield a solution
$U_n\in C^2(\Omega)\cap C(\overline\Omega)$
 {(see~\cite{Amann76} and~\cite[Chapter~3]{Pao92} for the method, and~\cite[Chapters~4 and~6]{GT}
for the Schauder estimates)}, with
\[
n\le U_n(x)\le n+\frac12\varphi(x)
\qquad(x\in\overline\Omega).
\]
For uniqueness, the difference $w$ of two solutions with boundary
value $n$ satisfies $-\Delta w=c(x)w$ with $c\le0$ and
$w=0$ on $\partial\Omega$. The  {weak} maximum principle
 {for the operator $\Delta+c$ with $c\le0$~\cite[Corollary~3.2]{GT}}, applied to
$w$ and $-w$, gives $w=0$.
} If $m>n$ and $w=U_m-U_n$, then $w=m-n$ on $\partial\Omega$ and
\[
-\Delta w=f(U_m)-f(U_n)=c(x)w,
\]
where one may take
\[
c(x)=\int_0^1 f'\bigl(U_n(x)+\theta w(x)\bigr)\,d\theta\le0.
\]
The  {weak maximum principle~\cite[Corollary~3.2]{GT}, applied to $-w$, gives $w\ge0$.
Since $w\not\equiv0$ and $\Omega$ is connected, the strong maximum
principle~\cite[Theorem~3.5]{GT} gives} $w>0$  {in $\Omega$}. Hence $f(U_m)\le f(U_n)$, so $-\Delta w\le0$; therefore $w$ is subharmonic and attains its maximum on $\partial\Omega$. Consequently,
\[
\|U_m-U_n\|_\infty=m-n.
\]
Thus $X=\{U_n:n\in\mathbb Z\}$, equipped with the supremum metric $d(u,v)=\|u-v\|_\infty$, is a complete discrete metric space isometric to $\mathbb Z$.

The functional $\mathcal H$ is order-preserving on the family $\{U_n\}$ and satisfies
\[
|\mathcal H[u] -\mathcal H[v]|
\le \frac{1}{4}\|u-v\|_\infty
\]
because $\psi$ is increasing and $1$-Lipschitz. Hence, for $m>n$,
\[
0\le \mathcal H[U_m]-\mathcal H[U_n]
\le \frac{1}{4}(m-n).
\]
Define
\[
F_b(n)=\left\lfloor b+\mathcal H[U_n]\right\rfloor,
\qquad
T_bU_n=U_{F_b(n)}.
\]
A function solves \eqref{eq:elliptic-nonlocal-problem} if and only if it is one of the $U_n$ and $n=F_b(n)$.

The map $F_b$ is nondecreasing and therefore has no periodic points of prime period two. Moreover, for $i<j<k$, with $D=k-i\ge2$, we have
$
P_d(U_i,U_j,U_k)=2D$,
and
\[
P_d(T_bU_i,T_bU_j,T_bU_k)
=2\bigl(F_b(k)-F_b(i)\bigr)
\le 2\left\lceil \frac{D}{4}\right\rceil
\le D
=\frac{1}{2}P_d(U_i,U_j,U_k).
\]
 {Here the first inequality follows from
\eqref{eq:rounding-increment}, applied with
$x=b+\mathcal H[U_i]$ and
$\eta=\mathcal H[U_k]-\mathcal H[U_i]\le D/4$, while the second
follows from $\lceil D/4\rceil\le\lceil D/3\rceil\le D/2$;
see~\eqref{eq:rounding-bounds}.}
Thus $T_b$ contracts perimeters of triangles on $X$ with coefficient $1/2$. By Theorem~\ref{t1}, problem \eqref{eq:elliptic-nonlocal-problem} has at least one and at most two solutions for every $b$.

There are values of $b$ for which the problem has exactly two solutions. Let
$\delta=\mathcal H[U_1]-\mathcal H[U_0]$.
Since $U_1>U_0$ on $\overline\Omega$ and $\psi$ is strictly increasing, $\delta>0$; by the above Lipschitz estimate, $\delta\le1/4$. Choose
\[
b_*=1-\frac{\mathcal H[U_0] + \mathcal H[U_1]}{2}.
\]
Then
\[
b_*+\mathcal H[U_0]=1-\frac{\delta}{2}\in[7/8,1),
\]
and
\[
b_*+\mathcal H[U_1] =1+\frac{\delta}{2}\in(1,9/8].
\]
Therefore $F_{b_*}(0)=0$ and $F_{b_*}(1)=1$, so $U_0$ and $U_1$ solve \eqref{eq:elliptic-nonlocal-problem}. Theorem~\ref{t1} excludes any further solution.

Finally, $T_b$ is not a Banach contraction in the supremum metric for any $b$. { 
Indeed, the bound $U_n\ge n$ and the monotonicity of $\psi$ give
\[
\mathcal H[U_n]
=\frac{1}{4|\Omega|}\int_\Omega\psi(U_n(x))\,dx
\ge\frac14\psi(n)\longrightarrow+\infty
\qquad(n\longrightarrow+\infty).
\]
} The same criterion as in Remark~\ref{lem:nonconstant-shooting} then shows that $T_b$ cannot be a Banach contraction. At $b=b_*$, the existence of two distinct fixed points rules out Banach contractivity with respect to any metric on $X$.
\end{example}

\begin{example}[A nonlinear heat equation with a nonlocal initial condition]
\label{ex:parabolic-pde}
Let $\Omega\subset\mathbb R^d$ be a bounded $C^2$ domain. Set again
\[
f(r)=\frac{1}{4}(1-\tanh r),\qquad
\psi(r)=\ln(1+e^r),
\]
and define
\[
\mathcal H[u]=\frac{1}{4}\int_0^1 \frac{1}{|\Omega|}
\int_\Omega \psi(u(t,x))\,dx\,dt.
\]
Consider the parabolic problem
\begin{equation}\label{eq:parabolic-nonlocal-problem}
\begin{cases}
\partial_t u=\Delta u+f(u), & 0<t\le1,\ x\in\Omega,\\[2mm]
\partial_\nu u=0, & 0<t\le1,\ x\in\partial\Omega,\\[2mm]
u(0,x)=\left\lfloor b+\mathcal H[u]\right\rfloor, & x\in\overline\Omega,
\end{cases}
\end{equation}
where the initial value is understood as a constant function of $x$. { 
A solution is understood as a function satisfying all three
relations pointwise and belonging to
\[
u\in C\bigl([0,1]\times\overline\Omega\bigr),\qquad
u\in C^{1,2}\bigl([\varepsilon,1]\times\overline\Omega\bigr)
\quad\text{for every }\varepsilon\in(0,1).
\]
Here the derivatives extend continuously to the lateral boundary
on each such time interval, derivatives at $t=1$ are understood
from the left, and $\nu$ is the outward unit normal. In particular,
the normal derivative in the boundary condition is well defined.
}

{ 
\noindent\emph{Physical interpretation.}
Here $u(t,x)$ is the temperature in a thermally insulated body
with the self-regulating heat source $f(u)$ of
Example~\ref{ex:elliptic-pde}, observed over one cycle $0\le t\le1$.
A digital controller uses the space--time average of the nonlinear
sensor response $\psi(u)$ to prescribe a spatially uniform initial
temperature for the next cycle in integer steps. Each cycle is
prepared externally at the prescribed temperature; insulation refers
to the evolution within the cycle. As in
Example~\ref{ex:nonlinear-nonlocal}, the nonlocal condition expresses
consistency of this rule: the initial temperature reproduces itself
from one cycle to the next.
\par}

For each $n\in\mathbb Z$, let $y_n$ be the unique solution of the scalar Cauchy problem
\[
y_n'(t)=f(y_n(t)),\qquad y_n(0)=n,
\]
and set
$
u_n(t,x)=y_n(t)$.
Since  {spatially constant functions satisfy $\Delta u=0$ in $\Omega$
and $\partial_\nu u=0$ on $\partial\Omega$}, each $u_n$ solves the PDE part of \eqref{eq:parabolic-nonlocal-problem}. { 
Conversely, let $u$ be a solution whose initial value is the
integer $n$, and set $w=u-u_n$. The specified regularity and
the homogeneous Neumann condition justify integration by parts
for $t>0$, giving
\[
\begin{aligned}
\frac12\frac{d}{dt}\int_\Omega w(t,x)^2\,dx
&=-\int_\Omega|\nabla w(t,x)|^2\,dx\\
&\quad+\int_\Omega w(t,x)
  \bigl(f(u(t,x))-f(u_n(t,x))\bigr)\,dx
\le0,
\end{aligned}
\]
because $f$ is nonincreasing. For $0<\varepsilon<t\le1$,
integration yields
\[
\int_\Omega w(t,x)^2\,dx
\le\int_\Omega w(\varepsilon,x)^2\,dx.
\]
Letting $\varepsilon\downarrow0$ and using continuity together
with $w(0,x)=0$, we conclude that $w=0$. Hence every solution
with integer initial value coincides with the corresponding $u_n$.
 {For uniqueness of classical solutions of semilinear parabolic problems
with Neumann boundary conditions via comparison principles,
see also~\cite[Chapter~2]{Pao92}; an abstract framework is given
in~\cite{Henry}.}
}

Equip $C([0,1]\times\overline\Omega)$ with the supremum metric
\[
d(u,v)=\max_{0\le t\le1}\max_{x\in\overline\Omega}|u(t,x)-v(t,x)|.
\]
By Lemma~\ref{lem:flow},
$
d(u_m,u_n)=|m-n|$.
Moreover, because the time interval has length one and the spatial integral is normalized,
\[
|\mathcal H[u] -\mathcal H[v] |
\le \frac{1}{4}\|u-v\|_\infty.
\]
For $m>n$, one has $u_m\ge u_n$, and therefore
\[
0\le \mathcal H[u_m] -\mathcal H[u_n]
\le \frac{1}{4}(m-n).
\]
Define
\[
F_b(n)=\left\lfloor b+\mathcal H[u_n]\right\rfloor,
\qquad
T_bu_n=u_{F_b(n)}.
\]
Then $F_b$ is nondecreasing and has no periodic points of prime period two.  {Since $0\le\mathcal H[u_k]-\mathcal H[u_i]\le D/4$ for
$i<k$ and $D=k-i$, the same computation as in
Example~\ref{ex:elliptic-pde} gives}
\[
P_d(T_bu_i,T_bu_j,T_bu_k)
\le \frac{1}{2}P_d(u_i,u_j,u_k)
\]
for all pairwise distinct $u_i,u_j,u_k\in\{u_n:n\in\mathbb Z\}$. By Theorem~\ref{t1}, problem \eqref{eq:parabolic-nonlocal-problem} has at least one and at most two solutions for every $b$.

For $b=3/4$, there are exactly two solutions. Indeed, the functions $u_0$ and $u_1$ are independent of $x$, and their functionals coincide with those in Example~\ref{ex:nonlinear-nonlocal}. Writing
\[
M_n=\int_0^1\psi(y_n(t))\,dt,
\]
the estimates from Example~\ref{ex:nonlinear-nonlocal} give
$0<M_0<1$ and  $1<M_1<2$.
Hence
\[
F_{3/4}(0)=\left\lfloor \frac34+\frac{M_0}{4}\right\rfloor=0,
\qquad
F_{3/4}(1)=\left\lfloor \frac34+\frac{M_1}{4}\right\rfloor=1.
\]
Thus $u_0$ and $u_1$ solve \eqref{eq:parabolic-nonlocal-problem}, and Theorem~\ref{t1} excludes any further solution.

Finally, for $n\ge1$, Lemma~\ref{lem:flow}  {with $L=1/4$, together with
$y_0(t)\ge0$ (since $f>0$),} gives
\[
y_n(t)\ge  {y_0(t)+}e^{-t/4}n\ge e^{-1/4}n,\qquad 0\le t\le1.
\]
Since $\psi(r)>r$ for  {all $r\in\mathbb R$}, it follows that
\[
\mathcal H[u_n]\ge \frac{e^{-1/4}}{4}n\to+\infty.
\]
Therefore Remark~\ref{lem:nonconstant-shooting} shows that $T_b$ is not a Banach contraction in the supremum metric for any $b$. At $b=3/4$, the two distinct fixed points additionally rule out Banach contractivity with respect to any metric on the discrete solution set.
\end{example}

The solutions in the parabolic example are spatially homogeneous. The elliptic example above provides spatially nonconstant solutions.  {In both examples, the fixed point operator contracts perimeters of triangles on a complete discrete invariant set and fails to be a Banach contraction in the specified supremum metric for every $b$. At the parameters for which two solutions exist, the corresponding operator cannot be a Banach contraction with respect to any metric on that set.}

\section{Applications to integer equations}
\label{sec:integer-equations}

Throughout this section, $I\subseteq\Z$ is an interval containing
at least three integers, equipped with $d(m,n)=|m-n|$.
Thus $I$ may be finite, a half-line, or all of $\Z$.
Write
\[
P(a,b,c)=|a-b|+|b-c|+|c-a|.
\]
It is clear that for any three real numbers the following equality holds
\begin{equation}\label{eq:perimeter-diameter}
P(a,b,c)=2\bigl(\max\{a,b,c\}-\min\{a,b,c\}\bigr).
\end{equation}
Notice that a Banach contraction $T\colon I\to I$ must be constant since
the inequality $|T(n+1)-T(n)|\le k<1$ forces
$T(n+1)=T(n)$ for every adjacent pair $\{n,n+1\}$ in $I$.

The following criterion reduces perimeter contractivity to a
condition on triples of consecutive integers.

\begin{theorem}
\label{thm:local}
For a mapping $T\colon I\to I$, the following conditions are equivalent:
\begin{itemize}
\item[$\mathrm{(i)}$] There exists $\alpha\in[0,1)$ such that
\begin{equation}\label{eq:perimeter-contraction}
P(Ta,Tb,Tc)\le\alpha P(a,b,c)
\end{equation}
for all pairwise distinct $a,b,c\in I$.
\item[$\mathrm{(ii)}$] For every $n$ with $n,n+2\in I$,
\begin{equation}\label{eq:local-diameter}
\max\{T(n),T(n+1),T(n+2)\}
-\min\{T(n),T(n+1),T(n+2)\}\le1.
\end{equation}
\end{itemize}
Whenever these conditions hold,
\begin{equation}\label{eq:pair-bound}
|T(i)-T(j)|\le\left\lceil\frac{|i-j|}{2}\right\rceil
\qquad(i,j\in I),
\end{equation}
and $2/3$ is an admissible perimeter contraction coefficient.
\end{theorem}

\begin{proof}
If $T$ is constant, then the assertions of the theorem are trivial. Throughout the proof assume that $T$ is nonconstant.

$\mathbf{(i)\Rightarrow(ii)}$. Let condition  {$\mathrm{(i)}$} hold. Applying \eqref{eq:perimeter-contraction} to $n,n+1,n+2$ and using~(\ref{eq:perimeter-diameter}) gives
\[
P\bigl(T(n),T(n+1),T(n+2)\bigr)\le {2\alpha((n+2)-n)}= 4\alpha<4.
\]
By \eqref{eq:perimeter-diameter}, 
\[
2(\max\{T(n),T(n+1),T(n+2)\}
-\min\{T(n),T(n+1),T(n+2)\})<4,
\]
\[
\max\{T(n),T(n+1),T(n+2)\}
-\min\{T(n),T(n+1),T(n+2)\}<2.
\]
This is equivalent to \eqref{eq:local-diameter}.

$\mathbf{(ii)}\Rightarrow$(\ref{eq:pair-bound}). Assume that $\mathrm{(ii)}$ holds. First, observe that every pair of integers
$p,q\in I$ satisfying
$1\le |p-q|\le 2$ is contained in some
triple of consecutive integers belonging to $I$.
Indeed, this follows from the fact that $I$
is an interval containing at least three integers.

Therefore, by \eqref{eq:local-diameter},
\begin{equation}\label{eq:local-proof-short}
|T(p)-T(q)|\le 1
\qquad
\text{whenever }1\le |p-q|\le 2.
\end{equation}

We now extend this estimate to arbitrary $i,j\in I$. Let $i<j$ and put $D=j-i$. By the division algorithm, there exist integers
$q\ge 0$ and $r\in\{0,1\}$ such that
\begin{equation}\label{eq:local-proof-division2}
D=2q+r.
\end{equation}
Consider the points
\[
i,\ i+2,\ i+4,\ \ldots,\ i+2q,\ j.
\]
All these points belong to $I$.
The distance between any two consecutive
distinct points in this sequence is either
one or two. Applying the triangle inequality, we get
\begin{align}
|T(i)-T(j)|
&\le
\sum_{k=0}^{q-1}
|T(i+2k)-T(i+2k+2)|
\nonumber\\
&\quad+
|T(i+2q)-T(j)|.
\label{eq:local-proof-chain2}
\end{align}

Each term in the sum in
\eqref{eq:local-proof-chain2} is at most one
by \eqref{eq:local-proof-short}.
Moreover, the last term is zero if $r=0$,
and is at most one if $r=1$.
Thus
\begin{equation}\label{eq:local-proof-bound2}
|T(i)-T(j)|
\le q+r
=
\left\lceil\frac{D}{2}\right\rceil.
\end{equation}

Since $D=j-i$, and the case $i>j$
follows by symmetry, we conclude that \eqref{eq:pair-bound} holds.

$\mathbf{(ii)\Rightarrow(i)}$. Let $a,b,c\in I$ be pairwise distinct.
Without loss of generality, suppose that
$a<b<c$ and put
\begin{equation}\label{eq:local-proof-D}
D=c-a\ge 2.
\end{equation}

Since
\[
|a-b|\le D,\qquad
|b-c|\le D,\qquad
|a-c|=D,
\]
inequality \eqref{eq:pair-bound} gives
\begin{align*}
|Ta-Tb|\le\left\lceil\frac D2\right\rceil,\quad
|Tb-Tc|\le\left\lceil\frac D2\right\rceil,\quad
|Ta-Tc|\le\left\lceil\frac D2\right\rceil.
\end{align*}

For a nonempty finite set $A\subseteq\R$, write $\operatorname{diam} A=\max A-\min A$.
By \eqref{eq:perimeter-diameter}, for every
$a<b<c$ in $I$ we have
\begin{equation}\label{eq:local-proof-ratio}
\frac{P(Ta,Tb,Tc)}{P(a,b,c)}
=
\frac{\operatorname{diam}\{Ta,Tb,Tc\}}{c-a}.
\end{equation}
Consequently,
\begin{equation}\label{eq:local-proof-diam2}
\operatorname{diam}\{Ta,Tb,Tc\}
\le
\left\lceil\frac D2\right\rceil.
\end{equation}
Combining \eqref{eq:local-proof-ratio},
\eqref{eq:local-proof-D}, and
\eqref{eq:local-proof-diam2}, we obtain
\begin{equation}\label{eq:local-proof-mainratio}
\frac{P(Ta,Tb,Tc)}{P(a,b,c)}
\le
\frac{\lceil D/2\rceil}{D}.
\end{equation}
 {Since $D\ge2$, inequalities
\eqref{eq:local-proof-mainratio} and
\eqref{eq:rounding-bounds} yield}
\begin{equation}\label{eq:local-proof-universal}
\frac{P(Ta,Tb,Tc)}{P(a,b,c)}
\le\frac23.
\end{equation}
Since $a,b,c$ were arbitrary pairwise
distinct points of $I$, we have proved
\begin{equation}\label{eq:local-proof-contract}
P(Ta,Tb,Tc)
\le\frac23 P(a,b,c).
\end{equation}
Hence $\mathrm{(i)}$ holds with $\alpha=2/3$.
\end{proof}

The local criterion also yields a quantitative version of finite
convergence to a fixed point.

\begin{proposition}
\label{prop:finite-termination}
Let $T\colon I\to I$ satisfy \eqref{eq:local-diameter}.
Then $T$ has a fixed point if and only if there is no $n,n+1\in I$
such that
\begin{equation}\label{eq:adjacent-cycle}
T(n)=n+1,\qquad T(n+1)=n.
\end{equation}
When these cycles are absent, $T$ has one or two fixed points,
and two distinct fixed points are consecutive integers.
Every orbit $x_{j+1}=T(x_j)$ reaches a fixed point after finitely
many steps. More precisely, if $r=|T(x_0)-x_0|\ge1$, a fixed
point is reached no later than
\begin{equation}\label{eq:integer-termination-bound}
x_{M+2},\qquad M=\lceil\log_2 r\rceil.
\end{equation}
If $r=0$, the initial point is already fixed.
\end{proposition}

\begin{proof}
Since $I\subseteq\mathbb Z$ is complete in the usual metric,
and Theorem~\ref{thm:local} shows that $T$ contracts perimeters
of triangles, Theorem~\ref{t1} is applicable.

We first prove the asserted fixed point criterion. Suppose that
$a,b\in I$ form a cycle of prime period two, that is,
$T(a)=b$, $T(b)=a$, and $a\ne b$.
Putting $d=|a-b|\ge1$, we obtain from \eqref{eq:pair-bound}
\begin{equation}\label{eq:ft-cycle}
d=|T(a)-T(b)|\le\left\lceil\frac d2\right\rceil.
\end{equation}
If $d\ge2$, then $\lceil d/2\rceil\le d-1$, which
contradicts \eqref{eq:ft-cycle}. Hence $d=1$, so $a$ and $b$
are consecutive integers. Writing $n=\min\{a,b\}$, we see
that the cycle has precisely the form \eqref{eq:adjacent-cycle}.
Conversely, \eqref{eq:adjacent-cycle} clearly defines a cycle
of prime period two. Thus $T$ has no periodic points of prime
period two if and only if \eqref{eq:adjacent-cycle} does not
hold for any $n,n+1\in I$. By Theorem~\ref{t1}, this is
equivalent to the existence of a fixed point.

Assume henceforth that no cycle of the form
\eqref{eq:adjacent-cycle} exists. By Theorem~\ref{t1},
$T$ has at least one and at most two fixed points.
If $p\ne q$ are fixed points, then, with $d=|p-q|$,
\eqref{eq:pair-bound} yields
\begin{equation}\label{eq:ft-fixed}
d=|T(p)-T(q)|\le\left\lceil\frac d2\right\rceil.
\end{equation}
As in \eqref{eq:ft-cycle}, this inequality forces $d=1$.
Therefore, $T$ has one or two fixed points, and any two
distinct fixed points are consecutive integers.

It remains to prove finite termination and the stated bound.
Let $x_0\in I$ and define $x_{j+1}=T(x_j)$ for $j\ge0$.
First suppose that $r=|T(x_0)-x_0|\ge1$, and put
\[
r_j=|x_{j+1}-x_j|\qquad(j\ge0),
\]
so that $r_0=r$. Applying \eqref{eq:pair-bound} to
$x_j$ and $x_{j+1}$, we obtain
\begin{align}
 r_{j+1}
 =|x_{j+2}-x_{j+1}|
   =|T(x_{j+1})-T(x_j)|
 \le\left\lceil\frac{|x_{j+1}-x_j|}{2}\right\rceil
   =\left\lceil\frac{r_j}{2}\right\rceil.
 \label{eq:ft-recursion}
\end{align}
We claim that
\begin{equation}\label{eq:ft-decay}
r_j\le\left\lceil\frac{r}{2^j}\right\rceil
\qquad(j\ge0).
\end{equation}
For $j=0$, this is an equality. If it holds for some $j$,
then \eqref{eq:ft-recursion} and the monotonicity of the
ceiling function give
\begin{align*}
r_{j+1}
\le\left\lceil\frac{r_j}{2}\right\rceil
 \le\left\lceil\frac12
       \left\lceil\frac{r}{2^j}\right\rceil\right\rceil
=\left\lceil\frac{r}{2^{j+1}}\right\rceil.
\end{align*}
 {The final equality follows from the nested-ceiling identity
\eqref{eq:rounding-nested} with $m=2$.}
This proves \eqref{eq:ft-decay} by induction.

Set $M=\lceil\log_2 r\rceil$. Since $2^M\ge r$,
\eqref{eq:ft-decay} implies
\begin{equation}\label{eq:ft-unit}
0\le r_M\le\left\lceil\frac{r}{2^M}\right\rceil\le1.
\end{equation}
If $r_M=0$, then $x_{M+1}=x_M$, so $x_M$ is a fixed point.
If $r_M=1$, then \eqref{eq:ft-recursion} gives
$r_{M+1}\le1$. When $r_{M+1}=0$, the point $x_{M+1}$
is fixed. Thus we only need to consider the case
\begin{equation}\label{eq:ft-two-steps}
r_M=r_{M+1}=1.
\end{equation}

By \eqref{eq:ft-two-steps}, both
$|x_{M+1}-x_M|$ and $|x_{M+2}-x_{M+1}|$ equal one.
If these two steps had opposite directions, we would have
$x_{M+2}=x_M$, and hence
\[
T(x_M)=x_{M+1},\qquad T(x_{M+1})=x_M,
\]
a forbidden cycle of prime period two.
Therefore $x_M,x_{M+1},x_{M+2}$ are three consecutive
integers in either increasing or decreasing order.
We may consequently apply \eqref{eq:local-diameter} to
these three points. Since
\[
T(x_M)=x_{M+1},\qquad T(x_{M+1})=x_{M+2},
\]
we obtain
\begin{equation}\label{eq:ft-final-diameter}
\max\{x_{M+1},x_{M+2},T(x_{M+2})\}
-\min\{x_{M+1},x_{M+2},T(x_{M+2})\}\le1.
\end{equation}
Now $x_{M+1}$ and $x_{M+2}$ are distinct integers at
distance one by \eqref{eq:ft-two-steps}.
Thus \eqref{eq:ft-final-diameter}, together with the fact
that $T(x_{M+2})\in\mathbb Z$, implies
\begin{equation}\label{eq:ft-final-options}
T(x_{M+2})\in\{x_{M+1},x_{M+2}\}.
\end{equation}
The first alternative in \eqref{eq:ft-final-options} is
impossible, since $T(x_{M+1})=x_{M+2}$ and
$x_{M+1}\ne x_{M+2}$ would produce a two-cycle.
Consequently $T(x_{M+2})=x_{M+2}$.

In all cases a fixed point is reached at one of
$x_M$, $x_{M+1}$, or $x_{M+2}$. Hence every orbit terminates
at a fixed point no later than $x_{M+2}$, as claimed
in \eqref{eq:integer-termination-bound}.

Finally, if $r=0$, then $T(x_0)=x_0$, so the initial point
is itself a fixed point. This completes the proof.
\end{proof}

We next apply the local criterion to integer equations
of the form $x=\lfloor f(x)\rfloor$. For nondecreasing
functions $f$, perimeter contractivity can be
characterized directly in terms of the rounded values
of $f$.

\begin{theorem}
\label{thm:integer}
Let $f\colon I\to\R$ be nondecreasing and suppose that
$T(n)=\lfloor f(n)\rfloor$ maps $I$ into $I$.
Then $T$ contracts perimeters of triangles if and only if
\begin{equation}\label{eq:rounded-two-step}
\lfloor f(n+2)\rfloor-\lfloor f(n)\rfloor\le1
\qquad(n,n+2\in I).
\end{equation}
Whenever \eqref{eq:rounded-two-step} holds, the equation
\begin{equation}\label{eq:integer-equation}
x=\lfloor f(x)\rfloor,\qquad x\in I,
\end{equation}
has one or two solutions, and two distinct solutions are consecutive.
\end{theorem}

\begin{proof}
Since $f$ is nondecreasing and the floor function
preserves order, $T$ is also nondecreasing. Thus, for
every $n,n+2\in I$,
\[
T(n)\le T(n+1)\le T(n+2),
\]
and consequently
\begin{align}
&\max\{T(n),T(n+1),T(n+2)\}
-\min\{T(n),T(n+1),T(n+2)\}
\nonumber\\
&\qquad=T(n+2)-T(n)
=\lfloor f(n+2)\rfloor-\lfloor f(n)\rfloor.
\label{eq:integer-proof-diameter}
\end{align}
By Theorem~\ref{thm:local}, $T$ contracts perimeters
of triangles if and only if the left-hand side of
\eqref{eq:integer-proof-diameter} does not exceed one.
In view of \eqref{eq:integer-proof-diameter}, this is
equivalent to \eqref{eq:rounded-two-step}.

Now suppose that \eqref{eq:rounded-two-step} holds.
Since $T$ is nondecreasing, it cannot have a cycle
of prime period two. Indeed, if $a<b$ and
$T(a)=b$, $T(b)=a$, then $b=T(a)\le T(b)=a$, which is impossible.

Therefore, Theorem~\ref{t1} guarantees that $T$
has at least one and at most two fixed points.
Moreover, by Proposition~\ref{prop:finite-termination},
any two distinct fixed points must be consecutive
integers. Since
$T(x)=x$
$\Longleftrightarrow$
$x=\lfloor f(x)\rfloor$,
the same conclusions hold for the solutions of
\eqref{eq:integer-equation}.
\end{proof}

\begin{example}[A nonlinear equation]\label{ex:nonlinear}
Consider the integer equation
\[
x=\left\lfloor\frac{x}{3}
+\frac{1}{12}\sin x\right\rfloor,
\qquad x\in\Z,
\]
and put
\[
f(t)=\frac{t}{3}+\frac{1}{12}\sin t.
\]
Since
\[
\frac14\le f'(t)
=\frac13+\frac1{12}\cos t
\le\frac5{12}
\qquad(t\in\R),
\]
the function $f$ is strictly increasing.

By the mean value theorem, for every $n\in\Z$
there exists $\xi_n\in(n,n+2)$ such that
\[
f(n+2)-f(n)=2f'(\xi_n).
\]
Using the above bounds on $f'$, we obtain
\begin{equation}\label{eq:nonlinear-two-step}
0<\frac12\le f(n+2)-f(n)
\le\frac56<1.
\end{equation}
 {Applying \eqref{eq:rounding-increment} with
$x=f(n)$ and $\eta=f(n+2)-f(n)\in(0,1)$, we obtain}
\[
0\le\lfloor f(n+2)\rfloor
-\lfloor f(n)\rfloor\le1
\qquad(n\in\Z).
\]
Thus condition~\eqref{eq:rounded-two-step}
is satisfied, and Theorem~\ref{thm:integer}
guarantees that the equation has at least one
and at most two integer solutions.

Finally, $f(0)=0$, while
\[
-1<-\frac5{12}
<f(-1)=-\frac13-\frac{\sin1}{12}<0.
\]
Hence
\[
\lfloor f(0)\rfloor=0,
\qquad
\lfloor f(-1)\rfloor=-1.
\]
Therefore, $0$ and $-1$ are solutions.
Since Theorem~\ref{thm:integer} allows at most
two solutions, these are the only ones.

{ 
\noindent\emph{Interpretation.}
The real map $t\mapsto f(t)$ is a contraction with the unique fixed point
$t=0$ and contraction coefficient $5/12$. The iteration
$x_{k+1}=\lfloor f(x_k)\rfloor$ implements it on the integer grid by
rounding every result toward $-\infty$, as in an idealized arithmetic
implementation or spatial discretization~\cite{DK93}. This rounding
creates the additional equilibrium $-1$, illustrating a dead-band
effect familiar from digital recursive filters~\cite{OS10}.
Theorem~\ref{thm:integer} shows that under
condition~\eqref{eq:rounded-two-step} there are at most two grid
equilibria, and they are adjacent. Here they are $-1$ and the original
equilibrium $0$.
\par}
\end{example}

The dynamics of the mappings
$x\mapsto\lfloor qx\rfloor$ and
$x\mapsto\lfloor b+qx\rfloor$
have been studied in~\cite{RSU16} and~\cite{H25},
respectively. The following proposition determines
the exact parameter range for which the latter
mappings contract perimeters of triangles and
provides an explicit description of their
integer fixed points.

\begin{proposition}
\label{thm:quantized-perimeter}
Let $b\in\R$, $q\ge0$, and define
\[
T(n)=\lfloor b+qn\rfloor,\qquad n\in\Z.
\]
Then $T$ contracts perimeters of triangles in the usual metric
if and only if $0\le q\le1/2$. 
For $q>0$, $T$ is not a Banach contraction in the usual metric.

For $0\le q\le1/2$, all solutions of
\begin{equation}\label{eq:affine-equation}
x=\lfloor b+qx\rfloor,\qquad x\in\Z,
\end{equation}
are exactly the integers satisfying
\begin{equation}\label{eq:affine-solutions}
\frac{b-1}{1-q}<x\le\frac b{1-q}.
\end{equation}
Their number is
\begin{equation}\label{eq:affine-count}
N(b,q)=\left\lfloor\frac b{1-q}\right\rfloor
-\left\lfloor\frac{b-1}{1-q}\right\rfloor\in\{1,2\}.
\end{equation}
\end{proposition}

\begin{proof}
First, suppose that $0\le q\le 1/2$. Since $q\ge0$, the mapping
$T$ is nondecreasing.  {Applying \eqref{eq:rounding-increment}
with $x=b+qn$ and $\eta=2q\in[0,1]$, we obtain}
\begin{equation}\label{eq:affine-proof-two-step}
0\le T(n+2)-T(n)\le1
\qquad(n\in\Z).
\end{equation}
In particular, $T(n)\le T(n+1)\le T(n+2)$, so
\begin{align*}
&\max\{T(n),T(n+1),T(n+2)\}
 -\min\{T(n),T(n+1),T(n+2)\}\\
&\qquad=T(n+2)-T(n)\le1.
\end{align*}
Thus condition~\eqref{eq:local-diameter} holds, and
Theorem~\ref{thm:local} shows that $T$ contracts the perimeters
of triangles.

Conversely, suppose that $T$ contracts the perimeters of triangles.
Since $T$ is nondecreasing, Theorem~\ref{thm:local} gives
\eqref{eq:affine-proof-two-step}. For any positive integer $m$,
summing this inequality over $n=0,2,\ldots,2m-2$ yields
\begin{align}
T(2m)-T(0)
=\sum_{j=0}^{m-1}\bigl(T(2j+2)-T(2j)\bigr)
\le m.
\label{eq:affine-proof-telescope}
\end{align}
 {On the other hand, the strict inequality in
\eqref{eq:rounding-difference}, applied with $x=b$ and
$y=b+2qm$, yields}
\begin{equation}\label{eq:affine-proof-lower}
T(2m)-T(0)
=\lfloor b+2qm\rfloor-\lfloor b\rfloor
>2qm-1.
\end{equation}
Combining \eqref{eq:affine-proof-telescope} and
\eqref{eq:affine-proof-lower}, we obtain
\begin{equation}\label{eq:affine-proof-qbound}
(2q-1)m<1
\qquad(m=1,2,\ldots).
\end{equation}
If $q>1/2$, the left-hand side of
\eqref{eq:affine-proof-qbound} becomes arbitrarily large as
$m\to\infty$, a contradiction. Therefore $q\le1/2$, which
completes the proof of the first assertion.

Next, let $q>0$. Since
\[
T(n)=\lfloor b+qn\rfloor>b+qn-1\longrightarrow+\infty
\qquad(n\to+\infty),
\]
the mapping $T$ is nonconstant. As observed at the beginning of
this section, a Banach contraction on $\Z$ with the usual metric
must be constant. Hence $T$ is not a Banach contraction.

Finally, assume $0\le q\le1/2$ and consider
\eqref{eq:affine-equation}. For an integer $x$, the defining
property of the floor function gives
\begin{equation}\label{eq:affine-proof-floor}
x=\lfloor b+qx\rfloor
\quad\Longleftrightarrow\quad
x\le b+qx<x+1.
\end{equation}
Since $1-q>0$, rearranging the inequalities in
\eqref{eq:affine-proof-floor} gives
\[
\frac{b-1}{1-q}<x\le\frac{b}{1-q},
\]
which proves \eqref{eq:affine-solutions}.

Put
\[
A=\frac{b-1}{1-q},\qquad B=\frac{b}{1-q}.
\]
The solutions are precisely the integers in $(A,B]$. Therefore,
by the elementary counting formula for a half-open interval,
\begin{equation}\label{eq:affine-proof-counting}
N(b,q)=\operatorname{card}\bigl((A,B]\cap\Z\bigr)
=\lfloor B\rfloor-\lfloor A\rfloor,
\end{equation}
which is exactly \eqref{eq:affine-count}.

It remains to justify that this number is either one or two.
Since $0\le q\le1/2$,
\begin{equation}\label{eq:affine-proof-length}
1\le B-A=\frac{1}{1-q}\le2.
\end{equation}
The interval $(A,B]$ contains at least one integer: indeed,
$m=\lfloor B\rfloor$ satisfies $m\le B$ and
\[
m>B-1\ge A,
\]
where the last inequality follows from
\eqref{eq:affine-proof-length}.
It cannot contain three distinct integers, since their smallest
and largest elements would differ by at least two, forcing
$B-A>2$, contrary to \eqref{eq:affine-proof-length}.
Consequently $N(b,q)\in\{1,2\}$, as claimed.
\end{proof}

{ 
\begin{remark}[Interpretation]
The map $n\mapsto\lfloor b+qn\rfloor$ is a first-order recursive filter
$x_{k+1}=qx_k+b$, with feedback coefficient $q$ and constant input $b$,
implemented on the integer grid by rounding toward $-\infty$.
Equivalently, it is an idealized linear filter with a uniform quantizer
of unit step in the loop~\cite{Delchamps90,OS10}.
For $0\le q\le1/2$, Proposition~\ref{thm:quantized-perimeter} describes
its stationary output levels completely. The steady state for the input $b$ is unique
unless $N(b,q)=2$, in which case two adjacent output levels are both
stationary. By Proposition~\ref{prop:finite-termination}, a trajectory
starting at $x_0$ with $r=|T(x_0)-x_0|\ge1$ reaches one of them in at most
$\lceil\log_2 r\rceil+2$ steps.
\end{remark}
\par}

\section{Applications to difference equations}\label{sec5}
Throughout this section, we write
\[
E=\ell_\infty(\mathbb Z)
=\left\{x=(x_n)_{n\in\mathbb Z}:
\|x\|_\infty:=\sup_{n\in\mathbb Z}|x_n|<\infty\right\}.
\]
Every subset of $E$ is equipped with the induced metric
$d(x,y)=\|x-y\|_\infty$, and we set
\[
P(x,y,z)=d(x,y)+d(y,z)+d(z,x).
\]

For the difference equation
\[
x_{n+1}=\frac13x_n-\frac43(-1)^n,
\qquad n\in\mathbb Z,
\]
both Theorem~\ref{t1} and Banach's fixed point theorem apply. Indeed, the associated operator on $\ell_\infty(\mathbb Z)$ is
\[
(\mathcal T x)_n=\frac13x_{n-1}+\frac43(-1)^n.
\]
For any bounded sequences $x,y$,
\begin{equation}\label{eq:pc-linear-contraction}
\begin{aligned}
d(\mathcal T x,\mathcal T y)=\|\mathcal T x-\mathcal T y\|_\infty
=\frac13\sup_{n\in\mathbb Z}|x_{n-1}-y_{n-1}|
=\frac13\|x-y\|_\infty=\frac13d(x,y).
\end{aligned}
\end{equation}
By~(\ref{eq:pc-linear-contraction}), $\mathcal T$ is a Banach contraction with contraction
coefficient $1/3$. The unique bounded two-sided solution is $x_n=(-1)^n,\qquad n\in\mathbb Z$.

In the following modified example on a complete invariant subset of $\ell_\infty(\mathbb Z)$, the associated
operator satisfies the hypotheses of Theorem~\ref{t1}, whereas the Banach contraction condition fails for every metric induced by a norm on $\ell_\infty(\mathbb Z)$.

\begin{example}\label{ex:pc-sine-unique}
Consider the difference equation
\begin{equation}\label{eq:pc-sine-unique}
x_{n+1}
=\frac13x_n-\frac43(-1)^n
-\frac{2}{3\sqrt3}
 \sin\!\left(\frac{2\pi}{3}\bigl(x_n-(-1)^n\bigr)\right),
\qquad n\in\mathbb Z.
\end{equation}
Its right-hand side is infinitely differentiable with respect to $x_n$.
Let us show that the equation has exactly one bounded real-valued solution $(x_n)_{n\in \mathbb Z}$,
namely
\begin{equation}\label{eq:pc-sine-solution}
x_n=(-1)^n,\qquad n\in\mathbb Z.
\end{equation}

Put
\[
g(t)=\frac t3-\frac{2}{3\sqrt3}\sin\frac{2\pi t}{3}.
\]
Then \eqref{eq:pc-sine-unique} takes the form
\begin{equation}\label{e3}
x_{n+1}=(-1)^{n+1}+g(x_n-(-1)^n).
\end{equation}
Define the associated operator by
\begin{equation}\label{e*}
(\mathcal T x)_n=(-1)^n+g(x_{n-1}-(-1)^{n-1}),
\qquad n\in\mathbb Z.
\end{equation}
Since $g$ maps bounded sets to bounded sets, $\mathcal T$ maps
$\ell_\infty(\mathbb Z)$ into itself. Its fixed points are precisely
the bounded solutions of \eqref{eq:pc-sine-unique}.

Consider the subset
\begin{equation}\label{e**}
X=\{x^{(k)}:k\in\mathbb Z\},
\qquad x^{(k)}_n=(-1)^n+k,
\end{equation}
equipped with the metric
\[
d(x,y)=\|x-y\|_\infty.
\]
For all $k,l\in\mathbb Z$,
\[
d(x^{(k)},x^{(l)})=|k-l|.
\]
Thus $(X,d)$ is isometric to $\mathbb Z$ with its usual metric.
In particular, it is complete and all its points are isolated.

For every $j\in\mathbb Z$,
\[
g(3j)=j,\qquad g(3j+1)=j,\qquad g(3j+2)=j+1.
\]
Consequently,
\begin{equation}\label{e***}
g(k)=\tau(k):=\left\lfloor\frac{k+1}{3}\right\rfloor,
\qquad k\in\mathbb Z.
\end{equation}
By~(\ref{e*}) and~(\ref{e**}) we have
\[
(\mathcal T x^{(k)})_n=(-1)^n+g((-1)^{n-1}+k-(-1)^{n-1})=(-1)^n+g(k)
\] 
and hence by~(\ref{e**}) and~(\ref{e***})
\begin{equation}\label{eq:pc-sine-invariance}
\mathcal T x^{(k)}=x^{(\tau(k))}.
\end{equation}
Therefore $\mathcal T(X)\subseteq X$.

Let $k<l<m$. Then, since $\tau$ is nondecreasing,
\begin{equation}\label{eq:pc-sine-perimeter-identities}
\begin{aligned}
P(x^{(k)},x^{(l)},x^{(m)})&=2(m-k),\\
P(\mathcal T x^{(k)},\mathcal T x^{(l)},\mathcal T x^{(m)})
&=2\bigl(\tau(m)-\tau(k)\bigr).
\end{aligned}
\end{equation}
 {For the integer $h=m-k\ge2$,
\eqref{eq:rounding-difference} and \eqref{eq:rounding-bounds} give}
\begin{equation}\label{eq:pc-sine-tau-bound}
\tau(m)-\tau(k)
\le\left\lceil\frac h3\right\rceil
\le\frac h2.
\end{equation}
By~(\ref{eq:pc-sine-perimeter-identities})
and~(\ref{eq:pc-sine-tau-bound}), we obtain
\begin{equation}\label{eq:pc-sine-perimeter}
P(\mathcal T x,\mathcal T y,\mathcal T z)
\le\frac12 P(x,y,z)
\end{equation}
for all pairwise distinct $x,y,z\in X$.

Moreover, $\tau$ has no points of prime period two. Indeed, its
monotonicity gives
\[
\tau(k)\geq k
\quad\Longrightarrow\quad
\tau^2(k)\geq\tau(k)\geq k,
\]
and
\[
\tau(k)\leq k
\quad\Longrightarrow\quad
\tau^2(k)\leq\tau(k)\leq k.
\]
Thus $\tau^2(k)=k$ implies $\tau(k)=k$.
By~(\ref{eq:pc-sine-invariance}), the same conclusion holds for
$\mathcal T|_X$.

By~(\ref{eq:pc-sine-perimeter}) and the absence of points of
prime period two, all the hypotheses of Theorem~\ref{t1} are
satisfied, so $\mathcal T$ has a fixed point in $X$. Since
\[
\tau(k)=k\quad\Longleftrightarrow\quad k=0,
\]
its unique fixed point in $X$ is $x^{(0)}$, given by~(\ref{eq:pc-sine-solution}).

The Banach contraction condition does not hold on $(X,d)$.
Indeed,
$\mathcal T x^{(1)}=x^{(0)}$,
$\mathcal T x^{(2)}=x^{(1)}$
and consequently
\begin{equation}\label{eq:pc-sine-not-banach}
d(\mathcal T x^{(1)},\mathcal T x^{(2)})
=1=d(x^{(1)},x^{(2)}).
\end{equation}
By~(\ref{eq:pc-sine-not-banach}), the Banach contraction inequality
cannot hold with a coefficient less than $1$. In fact,
\[
\mathcal T x^{(1)}-\mathcal T x^{(2)}
=x^{(1)}-x^{(2)}
=-\mathbf 1,
\]
where $\mathbf 1$ is the constant sequence with value $1$.
Thus the same obstruction occurs for every metric induced by a norm
on $\ell_\infty(\mathbb Z)$.

It remains to prove uniqueness among all bounded real-valued
solutions on $\mathbb Z$. We first establish
\begin{equation}\label{eq:pc-sine-radial}
|g(t)|\le\frac23|t|,
\qquad t\in\mathbb R.
\end{equation}
Since $g$ is odd, it suffices to consider $t\ge0$.
For $0\le t\le3/2$, the inequalities
$0\le\sin(2\pi t/3)\le2\pi t/3$ yield
\[
\left(\frac13-\frac{4\pi}{9\sqrt3}\right)t
\le g(t)\le\frac t3.
\]
As $4\pi<9\sqrt3$, these bounds imply
$|g(t)|\le2t/3$. For $t\ge3/2$,
\[
|g(t)|
\le\frac t3+\frac{2}{3\sqrt3}
\le\left(\frac13+\frac{4}{9\sqrt3}\right)t
<\frac23t.
\]
This proves \eqref{eq:pc-sine-radial}.

Let $x$ be any bounded solution of \eqref{eq:pc-sine-unique}, and set
$y_n=x_n-(-1)^n$. Then $y$ is bounded and by~(\ref{e3})
\[
y_{n+1}=g(y_n),\qquad n\in\mathbb Z.
\]
By~(\ref{eq:pc-sine-radial}), for every fixed $n\in\mathbb Z$
and every integer $j\ge1$, we have
\begin{equation}\label{eq:pc-sine-backward}
|y_n|
\le\left(\frac23\right)^j|y_{n-j}|
\le\left(\frac23\right)^j
   \sup_{m\in\mathbb Z}|y_m|.
\end{equation}
Letting $j\to\infty$ in~(\ref{eq:pc-sine-backward}), we obtain $y_n=0$.
Hence \eqref{eq:pc-sine-solution} is the unique bounded solution
on $\mathbb Z$.

{ 
\noindent\emph{Interpretation.}
Equation~\eqref{eq:pc-sine-unique} describes a damped discrete-time system
driven by the alternating input $(-1)^n$. The variable
$y_n=x_n-(-1)^n$ is the deviation from the forced response, and it obeys
the autonomous law $y_{n+1}=g(y_n)$. The nonlinearity $g$ is smooth, but
on the invariant lattice $\mathbb Z$ it coincides with the
quantizer $\tau(k)=\lfloor(k+1)/3\rfloor$. Thus a smooth system may behave on
an invariant set exactly as a quantized one. On the associated complete
set of sequence profiles, Theorem~\ref{t1} applies while the Banach
contraction inequality fails. The unique bounded solution has period
two, equal to the period of the forcing.
\par}
\end{example}

In the following example, we apply Theorem~\ref{t1} to establish the existence of exactly two solutions.
\begin{example}\label{ex:pc-floor-two}
Consider the difference equation
\begin{equation}\label{eq:pc-floor-two}
x_{n+1}=(-1)^{n+1}
+\left\lfloor\frac{x_n-(-1)^n}{2}\right\rfloor,
\qquad n\in\mathbb Z,
\end{equation}
where $\lfloor t\rfloor$ denotes the greatest integer not exceeding $t$.
Let us show that its only two bounded real-valued solutions on $\mathbb Z$ are
\begin{equation}\label{eq:pc-floor-solutions}
x_n=(-1)^n
\quad\text{and}\quad
x_n=(-1)^n-1,
\qquad n\in\mathbb Z.
\end{equation}
Moreover, on a complete invariant subset of $\ell_\infty(\mathbb Z)$, the associated operator satisfies the hypotheses of Theorem~\ref{t1}. It cannot satisfy the Banach contraction condition with respect to any metric on this subset.

Consider
\begin{equation}\label{r7}
X=\{x^{(k)}:k\in\mathbb Z\},
\qquad x^{(k)}_n=(-1)^n+k,
\qquad d(x,y)=\|x-y\|_\infty.
\end{equation}
As before,
\[
d(x^{(k)},x^{(l)})=|k-l|,
\]
so $(X,d)$ is complete.

The operator associated with \eqref{eq:pc-floor-two} is
\begin{equation}\label{r8}
(\mathcal T x)_n
=(-1)^n+\left\lfloor\frac{x_{n-1}-(-1)^{n-1}}2\right\rfloor.
\end{equation}
It maps bounded sequences to bounded sequences, and its fixed points
are exactly the bounded solutions of \eqref{eq:pc-floor-two}.
By~(\ref{r7}) and~(\ref{r8}) we have
\[
(\mathcal T x^{(k)})_n=(-1)^n+\left\lfloor\frac{(-1)^{n-1}+k-(-1)^{n-1}}2\right\rfloor =(-1)^n+\left\lfloor\frac k2\right\rfloor.
\]
Moreover, by ~(\ref{r7})
\[
(x^{(\tau(k))})_n=(-1)^n+\tau(k).
\]
Hence, for $k\in\mathbb Z$,
\begin{equation}\label{eq:pc-floor-invariance}
\mathcal T x^{(k)}=x^{(\tau(k))}, \,
\text{ where } \, \tau(k)=\left\lfloor\frac k2\right\rfloor.
\end{equation}
By~(\ref{eq:pc-floor-invariance}), $\mathcal T(X)\subseteq X$.

Let $k<l<m$. Then, by the monotonicity of $\tau$,
\begin{equation}\label{eq:pc-floor-perimeter-identities}
\begin{aligned}
P(x^{(k)},x^{(l)},x^{(m)})&=2(m-k),\\
P(\mathcal T x^{(k)},\mathcal T x^{(l)},\mathcal T x^{(m)})
&=2\bigl(\tau(m)-\tau(k)\bigr).
\end{aligned}
\end{equation}
 {For $h=m-k\ge2$,
\eqref{eq:rounding-difference} and \eqref{eq:rounding-bounds} yield}
\begin{equation}\label{eq:pc-floor-tau-bound}
\tau(m)-\tau(k)
\le\left\lceil\frac h2\right\rceil
\le\frac23h.
\end{equation}
By~(\ref{eq:pc-floor-perimeter-identities})
and~(\ref{eq:pc-floor-tau-bound}), we obtain
\begin{equation}\label{eq:pc-floor-perimeter}
P(\mathcal T x,\mathcal T y,\mathcal T z)
\le\frac23 P(x,y,z)
\end{equation}
for all pairwise distinct $x,y,z\in X$.

Since $\tau$ is nondecreasing, the argument from the preceding
example shows that
\[
\tau^2(k)=k\quad\Longrightarrow\quad\tau(k)=k.
\]
By~(\ref{eq:pc-floor-invariance}), $\mathcal T|_X$
has no points of prime period two.

By~(\ref{eq:pc-floor-perimeter}) and Theorem~\ref{t1},
$\mathcal T$ has a fixed point in $X$ and the number of its fixed
points does not exceed two. Here this upper bound is attained, because
\[
\tau(k)=k
\quad\Longleftrightarrow\quad
k\in\{-1,0\}.
\]
By~(\ref{eq:pc-floor-invariance}), the two sequences
in~(\ref{eq:pc-floor-solutions}) are exactly the fixed points
of $\mathcal T$ in $X$.

We next show that there are no other bounded real-valued solutions
on $\mathbb Z$. Let $x$ be any such solution, and put $y_n=x_n-(-1)^n$.
Then
\[
y_{n+1}=\left\lfloor\frac{y_n}{2}\right\rfloor.
\]
Iterating this identity gives
\begin{equation}\label{eq:pc-floor-backward}
y_n=\left\lfloor\frac{y_{n-j}}{2^j}\right\rfloor,
\qquad n\in\mathbb Z,\quad j\ge1.
\end{equation}
The iteration uses the nested-floor identity in
\eqref{eq:rounding-nested}, applied repeatedly with $m=2$.
Since $y$ is bounded, for sufficiently large $j$ we have
\[
-1<\frac{y_{n-j}}{2^j}<1.
\]
By~(\ref{eq:pc-floor-backward}), we obtain
$y_n\in\{-1,0\}$ for every $n\in\mathbb Z$.
Both values are fixed by $\tau$, so $y_{n+1}=y_n$ for every $n$.
This proves that \eqref{eq:pc-floor-solutions} lists all bounded
solutions on $\mathbb Z$.

Finally, let $u=x^{(-1)}$ and $v=x^{(0)}$.
These are distinct fixed points of $\mathcal T$.
If, for some metric $\rho$ on $X$ and some $q<1$, the inequality
\[
\rho(\mathcal T x,\mathcal T y)\le q\,\rho(x,y)
\]
held for all $x,y\in X$, then
\[
0<\rho(u,v)
=\rho(\mathcal T u,\mathcal T v)
\le q\,\rho(u,v),
\]
a contradiction. Thus no metric on $X$ can make $\mathcal T|_X$
a Banach contraction.

{ 
\noindent\emph{Interpretation.}
For the deviation $y_n=x_n-(-1)^n$, equation~\eqref{eq:pc-floor-two}
becomes $y_{n+1}=\lfloor y_n/2\rfloor$. For integer states, this is
halving with rounding toward $-\infty$, implemented by an arithmetic
right shift by one bit in two's complement representation. This
idealized filter illustrates a dead-band effect caused by
rounding~\cite{OS10}: the nonzero level $-1$ is stationary because
$\lfloor-1/2\rfloor=-1$. The direction of rounding matters, since
truncation toward zero would map $-1/2$ to $0$. Under the alternating
forcing, the two bounded solutions in~\eqref{eq:pc-floor-solutions}
are the two periodic responses of the original system.
\par}
\end{example}

\end{document}